%% file: main.tex
\documentclass[12pt,letterpaper]{amsart}
\usepackage{amsfonts,amssymb,amsmath,amscd,amsthm,enumerate,hyperref,cleveref}
\usepackage[alphabetic]{amsrefs}
\usepackage[english]{babel}
\usepackage{bm}
\usepackage{comment}

\usepackage{blindtext}
\usepackage{subfiles}

\usepackage{pgf,tikz}
\usepackage{mathrsfs}
\usetikzlibrary{arrows}
\usetikzlibrary{positioning}

\usepackage{graphicx}

\usepackage{color}

\usepackage[margin=1.1in]{geometry}
\newtheorem{theorem}{Theorem}[section]
\newtheorem*{theorem*}{Theorem}
\newtheorem{lemma}[theorem]{Lemma}
\newtheorem{proposition}[theorem]{Proposition}
\newtheorem{conjecture}[theorem]{Conjecture}
\newtheorem{corollary}[theorem]{Corollary}

\theoremstyle{definition}
\newtheorem{definition}[theorem]{Definition}
\newtheorem{rmk}{Remark}

\newcommand{\E}{\mathbb{E}}

\newcommand{\R}{\mathbb{R}}
\newcommand{\N}{\mathbb{N}}
\newcommand{\Z}{\mathbb{Z}}

\newcommand{\cD}{\mathcal{D}}
\newcommand{\cM}{\mathcal{M}}
\newcommand{\ctM}{\widetilde{\mathcal{M}}}
\newcommand{\ctD}{\widetilde{\mathcal{D}}}

\newcommand{\cR}{\mathcal{R}}

\newcommand{\rh}{\mathfrak{r}}
\renewcommand{\i}{\mathbf{i}}
\newcommand{\ve}{\varepsilon}

\newcommand{\rmU}{\mathrm{U}}
\newcommand{\GT}{\mathbb{GT}}

\newcommand\numberthis{\addtocounter{equation}{1}\tag{\theequation}}

\newcommand{\del}{\partial}
\newcommand{\cyc}{\sum_{\mathrm{cyc}}}
\DeclareRobustCommand{\stirling}{\genfrac\{\}{0pt}{}}

\title{A Quantized Analogue of the Markov-Krein Correspondence}
\author{Gopal K. Goel}
\address{Portland, OR}
\email{gopal.krishna.goel@gmail.com}
\author{Andrew Yao}
\address{Boston, MA}
\email{andrew.j.yao@gmail.com}

\begin{document}

\begin{abstract}
We study a family of measures originating from the signatures of the irreducible components of representations of the unitary group, as the size of the group goes to infinity. Given a random signature $\lambda$ of length $N$ with counting measure $\mathbf{m}$, we obtain a random signature $\mu$ of length $N-1$ through projection onto a unitary group of lower dimension. The signature $\mu$ interlaces with the signature $\lambda$, and we record the data of $\mu,\lambda$ in a random rectangular Young diagram $w$. We show that under a certain set of conditions on $\lambda$, both $\mathbf{m}$ and $w$ converge as $N\to\infty$. We provide an explicit moment generating function relationship between the limiting objects. We further show that the moment generating function relationship induces a bijection between bounded measures and certain continual Young diagrams, which can be viewed as a quantized analogue of the Markov-Krein correspondence. 
\end{abstract}

\maketitle

\section{Introduction}

Given a random matrix $M_N$ whose distribution is invariant under conjugation by unitary matrices, let $\lambda$ be the random vector of its eigenvalues and $\mu = \pi_{N,N-1}\lambda$ be the random vector of the eigenvalues of a principal $(N-1)\times(N-1)$ submatrix of $M_N$. We begin with a folk theorem in random matrix theory, stated here as a conjecture:
\begin{conjecture}
\label{con:MK_intro}
Suppose that we have a sequence of unitarily invariant random matrices $M_N$ such that the spectral measure $\frac{1}{N}\sum_{i=1}^N \delta_{\lambda_i}$ converges weakly in probability to a deterministic measure $\mathbf{m}_{\mathrm{RMT}}$ as $N\to\infty$. Let $\mu=\pi_{N,N-1}\lambda$. The random measure $\sum_{i=1}^N \delta_{\lambda_i} - \sum_{i=1}^{N-1} \delta_{\mu_i}$ converges weakly in probability to a signed measure $\mathbf{d}_{\mathrm{RMT}}$ which is related to $\mathbf{m}_{\mathrm{RMT}}$ by
\begin{equation} \label{eq:MK_intro}
\exp\left(\sum_{k=1}^\infty\frac{z^k}{k}\int_\mathbb{R}x^k \mathbf{d}_{\mathrm{RMT}}(dx)\right)=\sum_{k=0}^\infty z^k\int_\mathbb{R} x^k\mathbf{m}_{\mathrm{RMT}}(dx).
\end{equation}
\end{conjecture}

\begin{figure}[ht]
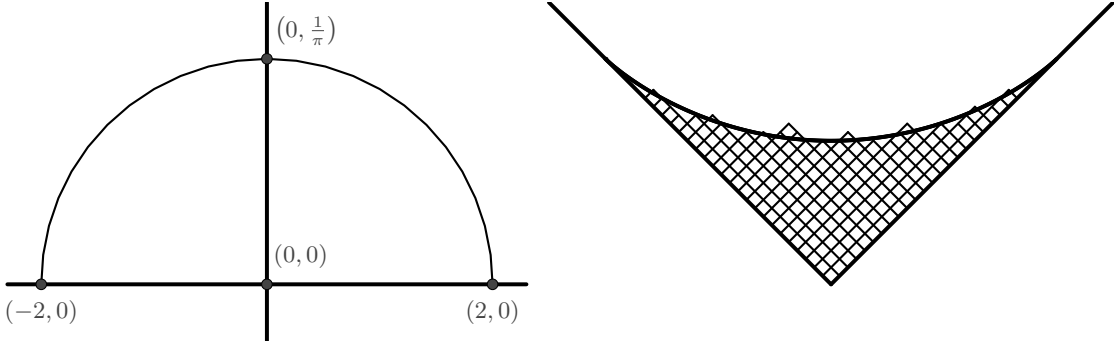

\subfile{SemicircleVKLSCurve}
\caption{Scaled Semicircle Law (left) and Vershik-Kerov-Logan-Shepp Curve, the limit of the Plancherel measure on Young diagrams (right)}
\label{fig:semicircle}
\end{figure}

Furthermore, the above relation is directly related to the Markov-Krein correspondence, which is a bijection between probability measures $\mathbf{m}$ and certain continual Young diagrams, observed by Kerov \cite{KEROV}. Here, $\mathbf{d}$ is the second derivative of the continual Young diagram, see also \cite{KREIN} and Theorem \ref{thm:BMK} in this text. A fascinating instance of this bijection, discovered by Kerov \cite{KEROV}, is where $\mathbf{m}$ is given by the semicircle law and $\mathbf{d}$ is the second derivative of the Vershik-Kerov-Logan-Shepp (VKLS) curve. We note that the semicircle law is the limiting spectral law of the Gaussian Unitary Ensemble (GUE) and more generally of Wigner matrices, see \cite{AGZ}*{Section 2.1}. On the other hand, the VKLS curve arises as the limiting diagram of the Plancherel measure on Young diagrams, see \cite{VK}, \cite{LOGAN}, and \Cref{fig:semicircle}.

Though a proof of \Cref{con:MK_intro} is unavailable in the literature, \cite{Bu} proved the special case where $M_N$ is an $N\times N$ GUE matrix, and more generally a Wigner matrix. Due to the lack of unitary invariance for general Wigner matrices, this suggests that the assumption of unitary invariance may be relaxed, though it is not clear how much one can relax the hypothesis.

The main result of this article establishes a quantized analogue of \Cref{con:MK_intro} where matrices are replaced with representations, and eigenvalues with signatures. Moreover, we find a quantized version of the Markov-Krein correspondence which is directly linked to our main result. We now introduce the quantized setting and our main result.

Let $\rmU(N)$ denote the $N$-dimensional unitary group. Recall that the irreducible representations of $\rmU(N)$ are in bijection with the set $\GT_N$ of $N$-tuples of non-increasing integers $\lambda = (\lambda_1, \ldots,\lambda_N)$ called \emph{signatures of length $N$}. Let $V_N^\lambda$ denote the irreducible representation corresponding to $\lambda \in \GT_N$. Given an arbitrary finite-dimensional representation $V_N$ of $\rmU(N)$, we define a probability measure $\rho[V_N]$ on $\GT_N$ where
\[ \rho[V_N](\lambda) = \frac{m_\lambda\dim V_N^\lambda}{\dim V_N}, \quad V_N = \bigoplus_{\lambda \in \GT_N} m_\lambda V_N^\lambda. \]
In other words, the probability weight of $\lambda$ is proportional to the dimension of the isotypic component corresponding to $\lambda$. For $\lambda \sim \rho[V_N]$, we define $\pi_{N,N-1}\lambda$ to be the random element of $\GT_{N-1}$ such that the probability distribution of $\pi_{N,N-1} \lambda$ given $\lambda$ is $\rho[V_N^\lambda|_{\rmU(N-1)}]$. In particular, the marginal distribution of $\pi_{N,N-1}\lambda$ is given by $\rho[V_N|_{\rmU(N-1)}]$.

Our main result is:
\begin{theorem}
\label{thm:DMK_intro}
Suppose that we have a sequence of representations $V_N$ of $\rmU(N)$ with $\lambda\sim\rho[V_N]$ such that distributions $\{\rho[V_N]\}_{N\ge 1}$ satisfy the technical assumption in Definition \ref{def:LLNA}. Then, the counting measure $\frac{1}{N}\sum_{i=1}^N \delta_{\lambda_i+N-i}$ converges weakly in probability to a deterministic measure $\mathbf{m}$ as $N\to\infty$, and the random measure $\sum_{i=1}^N \delta_{\lambda_i+N-i} - \sum_{i=1}^{N-1} \delta_{\mu_i+N-1-i}$, where $\mu=\pi_{N,N-1}\lambda$, converges weakly in probability to a signed measure $\mathbf{d}$ which are related by
\begin{equation} \label{eq:QMK_intro}
\exp\left(\sum_{k=1}^\infty\frac{z^k}{k}\int_\mathbb{R}x^k \mathbf{d}(dx)\right)=\frac{1}{z}\left(-1+\exp\left(z\sum_{k=0}^\infty z^k\int_\mathbb{R}x^k \mathbf{m}(dx)\right)\right).
\end{equation}
\end{theorem}
Furthermore, we show that the above relation defines a bijection between probability measures with density bounded by $1$ and a certain subclass of continual Young diagrams, where $\mathbf{m}$ is the bounded measure, and $\mathbf{d}$ can be viewed as the second derivative of the continual Young diagram. An exact statement is given in Theorem \ref{thm:BQMK}. This bijection is a quantized analogue of the Markov-Krein correspondence. Moreover, the Markov-Krein correspondence can be obtained from our quantized correspondence through a semiclassical limit, see Section \ref{ssec:semiclassical} for details.

Just as the semicircle law and the VKLS curve are linked by the Markov-Krein correspondence, we similarly find that the one-sided Plancherel measures are linked to the VKLS curve by our quantized correspondence, see \Cref{sec:appendixb}. Here, the one-sided Plancherel measures are a one-parameter family of distributions describing $N\to\infty$ limits of certain characters of $\rmU(N)$, see \cite{BBO}.

We expect that \Cref{con:MK_intro} may be proved by adapting the proof of \Cref{thm:DMK_intro}. Since the focus of this article is on the quantized setting, we do not pursue this direction, though we point out the necessary adaptations, see the end of \Cref{ssec:random_signatures}.

The main tool we use is the Fourier transform on representations of the unitary group and differential operators whose eigenfunctions are given by Schur functions. The action of these operators on the Fourier transform yield combinatorial expressions for the moments of the measures $\rho[V_N]$ and $\rho[V_N^\lambda|_{\mathrm{U}(N-1)}]$, from which we can study their behavior in the large $N$ limit. The use of the Fourier transform to study large $N$ limits of representations of the unitary group was pioneered by Bufetov and Gorin in \cite{BuG1}, where they studied limit shapes for the classical Lie groups. Their methods were further developed in \cites{BuG2,BuG3} to study global fluctuations for discrete particle systems, encompassing a variety of applications including large limits of lozenge tilings, domino tilings, and representations of the unitary group. Related methods were also recently used to obtain large $N$ local asymptotics of measures $\rho[V_N](\lambda)$ \cite{AHN}. For our purposes, we require the expansion of the moments of $\rho[V_N]$ to subleading terms. A critical property that we use for our analysis is that the differential operators commute with each other asymptotically relative to the order of the limit.

Using this approach, it should be possible to further refine our results to study the global fluctuations of the signed measure $\sum_{i=1}^N \delta_{\lambda_i+N-i} - \sum_{i=1}^{N-1} \delta_{\mu_i+N-1-i}$, though we do not pursue this here. We note that the global fluctuations of the random matrix analogue of these signed measures were studied in \cite{ERDOS} for Wigner matrices, and identified with a derivative of the Gaussian free field --- a $2$ dimensional conformally invariant, universal random distribution. Similar results were also established for the $\beta$-Jacobi corners process \cite{GZ} and the $\beta$-Hermite ensemble \cite{AG}.

The article is organized as follows. In \Cref{sec:main results}, we set up and state our main results, \Cref{thm:DMK}, a technical modification of \Cref{thm:DMK_intro}, and \Cref{thm:BQMK}, the quantized Markov-Krein bijection, as well as touch upon their respective random matrix analogues. In \Cref{sec:momlim}, we set up and prove \Cref{thm:momreg,thm:momdiff}, which explicitly calculate the moments referenced in \Cref{thm:DMK}. In \Cref{sec:theproof}, we use \Cref{thm:momreg,thm:momdiff} to prove \Cref{thm:DMK}. In \Cref{sec:bijection_proof}, we prove \Cref{thm:BQMK} and show how we can extract the Markov-Krein correspondence through a semiclassical limit. Finally, in \Cref{sec:appendixb}, we give an example of objects paired by the quantized Markov-Krein correspondence.

\textbf{Acknowledgements}. We would like to thank our mentor Andrew J. Ahn for guiding us throughout this project. Also, we would like to thank Vadim Gorin for suggesting the project and providing comments.

\section{Main Results and Random Matrix Analogues}
\label{sec:main results}

There are two main results we prove. The first is a technical restatement of Theorem \ref{thm:DMK_intro}. The second main result establishes that the relation \eqref{eq:QMK_intro} gives a bijective correspondence between a class of continual Young diagrams with second derivative $\mathbf{d}$ and probability measures $\mathbf{m}$. For further motivation, we also present random matrix analogues of these results.

\subsection{Asymptotics of Random Signatures} \label{ssec:random_signatures}

An $N$-tuple of non-increasing positive integers $\lambda=(\lambda_1\ge\lambda_2\ge\cdots\ge\lambda_N)$ is called a \textit{signature} of length $N$, and we denote by $\GT_N$ the set of all signatures of length $N$. The \textit{Schur function} $s_\lambda$, for $\lambda\in\GT_N$, is a symmetric rational function defined by
\[s_\lambda(x_1,\ldots,x_N)=\frac{\det\left[x_i^{\lambda_j+N-j}\right]_{i,j=1}^N}{\prod_{i<j}(x_i-x_j)}.\]
Let $\rh$ be a probability measure on the set $\GT_N$. The \textit{Schur generating function} $S_\rh(x_1,\ldots,x_N)$ of $\rh$ is a symmetric Laurent power series in $x_1,\ldots,x_N$ given by
\[S_\rh(x_1,\ldots,x_N)=\sum_{\lambda\in\GT_N}\rh(\lambda)\frac{s_\lambda(x_1,\ldots,x_N)}{s_\lambda(1^N)}.\]
In general, we will assume that the measure $\rh$ is such that this sum is uniformly convergent in an open neighborhood of $(1^N)$.

A key object of study is the \textit{projection map} $\pi_{N,N-1}$, depending on context either a map from random variables on $\GT_N$ to random variables on $\GT_{N-1}$, or a map from probability distributions on $\GT_N$ to probability distributions on $\GT_{N-1}$.
\begin{definition}
\label{def:proj}
For any fixed $\lambda\in\GT_N$, let $V_N^\lambda$ denote the corresponding irreducible representation of $\rmU(N)$. Suppose that its restriction onto $\mathrm{U}(N-1)$ decomposes into irreducible representations as
\[V_N^\lambda|_{\rmU(N-1)} = \bigoplus_{\mu \in \GT_{N-1}} m_\mu V_{N-1}^\mu. \]
Now, if $\lambda$ is a random variable, then $\pi_{N,N-1}\lambda$ is the random element of $\GT_{N-1}$ such that
\[ \mathbb{P}(\pi_{N,N-1}\lambda = \mu |\lambda) = \frac{m_\mu\dim(V_{N-1}^\mu)}{\dim (V_N^\lambda|_{\rmU(N-1)})}. \]
\end{definition}

Our goal is to study asymptotics of counting measures of random signatures. To this end, for any $\lambda\in\GT_N$, we define the counting measure
\[m[\lambda]:=\frac{1}{N}\sum_{i=1}^N\delta\left(\frac{\lambda_i+N-i}{N}\right).\]
Furthermore, for any $\lambda\in\GT_N$ and $\mu\in\GT_{N-1}$, we define
\[d[\lambda,\mu]:=\sum_{i=1}^N\delta\left(\frac{\lambda_i+N-i}{N}\right)-\sum_{i=1}^{N-1}\delta\left(\frac{\mu_i+N-1-i}{N}\right).\]
The pushforward of a measure $\rh$ on $\GT_N$ with respect to the map $\lambda\to m[\lambda]$ defines a random measure on $\mathbb{R}$ which we denote by $m[\rh]$. Furthermore, the random measure $d[\rh]$ on $\mathbb{R}$ is defined as $d[\lambda,\mu]$ where $\lambda$ is distributed according to $\rh$ and $\mu$ is distributed according to $\pi_{N,N-1}\lambda$.

We will be interested in asymptotics as $N\to\infty$, so we have the following definition.

\begin{definition}[\cite{BuG2}*{Definition 2.1}]
\label{def:LLNA}
A sequence of symmetric functions $\{F_N(x_1,\ldots,x_N)\}_{N\ge 1}$ is called \textbf{LLN-appropriate} if there exists a collection of reals $\{c_k\}_{k\ge 1}$ such that
\begin{itemize}
    \item For any $N$, the function $\log F_N(x_1,\ldots,x_N)$ is holomorphic in an open complex neighborhood of $(1^N)$.
    \item For any index $j$ and any $k\in\mathbb{N}$ we have
    \[\lim_{N\to\infty}\left.\frac{\del_j^k\log F_N(x_1,\ldots,x_N)}{N}\right|_{x_i=1}=c_k.\]
    \item For any $s\in\mathbb{N}$ and any indices $i_1,\ldots,i_s$ such at least two are distinct, we have
    \[\lim_{N\to\infty}\left.\frac{\del_{i_1}\cdots\del_{i_s}\log F_N(x_1,\ldots,x_N)}{N}\right|_{x_i=1}=0.\]
    \item The power series
    \[\sum_{k=1}^\infty\frac{c_k}{(k-1)!}(x-1)^{k-1}\]
    converges in a neighborhood of unity.
\end{itemize}
Now, a sequence $\rho=\{\rho_N\}_{N\ge 1}$, where $\rho_N$ is a probability measure on $\GT_N$, is called \textbf{LLN-appropriate} if the sequence $\{S_{\rho_N}\}_{N\ge 1}$ of its Schur generating functions is LLN-appropriate. For this sequence, we let $H_\rho(x)$ be a holomorphic function in a neighborhood of unity such that
\[H_\rho'(x) = \sum_{k=1}^\infty\frac{c_k}{(k-1)!}(x-1)^{k-1}.\]
\end{definition}

We have the following technical restatement of Theorem \ref{thm:DMK_intro}, which is what we will refer to for the rest of the paper.
\begin{theorem}
\label{thm:DMK}
Suppose that a sequence of probability measures $\rho=\{\rho_N\}_{N\ge 1}$, where $\rho_N$ is a probability measure on $\GT_N$, is LLN-appropriate. Then the random measures $m[\rho_N]$ and $d[\rho_N]$ converge as $N\to\infty$ in probability, in the sense of moments to \textit{deterministic} measures $\mathbf{m}$ and $\mathbf{d}$ on $\mathbb{R}$ respectively. Furthermore, their moment generating functions are related by
\begin{equation}
\label{eq:DMK}
\exp\left(\sum_{k=1}^\infty\frac{z^k}{k}\int_\mathbb{R}x^k \mathbf{d}(dx)\right)=\frac{1}{z}\left(-1+\exp\left(z\sum_{k=0}^\infty z^k\int_\mathbb{R}x^k \mathbf{m}(dx)\right)\right).
\end{equation}
\end{theorem}
Here are two applications demonstrating that the natural operations of tensor products and projections give LLN-appropriate sequences, meaning Theorem \ref{thm:DMK} applies to them. We adopt the notion below from \cite{BuG1}.
\begin{definition}[\cite{BuG1}*{Definition 2.5}] \label{def:reg}
A sequence $\lambda(N)$ of signatures is called \textit{regular} if there is a piecewise continuous function $f(t)$ and a constant $C$ such that
\[\lim_{N\to\infty}\sum_{j=1}^N\left|\frac{\lambda_j(N)}{N}-f(j/N)\right|=0\]
and
\[\left|\frac{\lambda_j(N)}{N}-f(j/N)\right|<C\]
for all $j=1,\ldots,N$ and $N=1,2,\ldots$.
\end{definition}
For a sequence of representations $V_N$, recall the measure $\rho[V_N]$ on $\GT_N$ defined in the introduction. Note that the Schur generating function of $\rho[V_N]$ is the normalized character of $V_N$, see the proof of \Cref{lem:ESGF} for justification.

The following result is about tensor products of irreducible representations.
\begin{corollary}
\label{thm:tensor}
Suppose $\lambda^{(1)}(N),\ldots,\lambda^{(r)}(N)$ are regular sequences of signatures. Let $V_N$ be the representation of $\rmU(N)$ given by
\[V_N=V_N^{\lambda^{(1)}(N)}\otimes\cdots\otimes V_N^{\lambda^{(r)}(N)}.\]
As $N\to\infty$, the measures $m[\rho[V_N]]$ and $d[\rho[V_N]]$ converge to measures $\mathbf{m}$ and $\mathbf{d}$ that are related by (\ref{eq:DMK}). 
\end{corollary}
\begin{proof}
Note that the character of $V_N^{\lambda^{(1)}(N)}\otimes\cdots\otimes V_N^{\lambda^{(r)}(N)}$ is simply the product of the characters of the $V_N^{\lambda^{(i)}(N)}$, so
\[S_{\rho[V_N]}(x_1,\ldots,x_N)=\prod_{i=1}^r\frac{s_{\lambda^{(i)}(N)}(x_1,\ldots,x_N)}{s_{\lambda^{(i)}(N)}(1^N)}.\]
It is well known that Schur functions of a regular sequence of signatures are LLN-appropriate (see Theorem 8.1 from \cite{BuG2}*{Theorem 8.1} for a reference). Furthermore, it is easy to see that a product of LLN-appropriate functions is also LLN-appropriate, so the conclusion of Theorem \ref{thm:DMK} holds here.
\end{proof}
The next result is about projections of irreducible representations.
\begin{corollary}
\label{thm:project}
Suppose we have a regular sequence $\lambda(N)$ of signatures, and some fixed $0<\eta<1$. Let $V_N$ be the representation of $\rmU(\lfloor\eta N\rfloor)$ given by 
\[V_N=\rho[V_N^\lambda|_{\rmU(\lfloor\eta N\rfloor)}]].\]
As $N\to\infty$, the measures $m[\rho[V_N]]$ and $d[\rho[V_N]]$ converge to measures $\mathbf{m}$ and $\mathbf{d}$ that are related by (\ref{eq:DMK}). 
\end{corollary}
\begin{proof}
It is easy to see that the character of $V_N$ is the character of $V_N^\lambda$ with $1$s plugged in for the last $N-\lfloor\eta N\rfloor$ entries, so the Schur generating function of $\rho[V_N]$ is simply
\[S_{\rho[V_N]}(x_1,\ldots,x_{\eta N})=\frac{s_\lambda(x_1,\ldots,x_{\lfloor{\eta N\rfloor}},1^{N-\lfloor\eta N\rfloor})}{s_\lambda(1^N)}.\]
It is easy to see that the restrictions of LLN-appropriate functions are also LLN-appropriate, so the conclusion of Theorem \ref{thm:DMK} holds in this setting.
\end{proof}

We now give some details on the random matrix analogue of these results, avoiding precise technical statements. Let $(M_N)_{N\ge 1}$ be a sequence of unitarily invariant random matrices --- that is, the distribution of $M_N$ is invariant under conjugation by fixed unitary matrices. Letting $\mathbb{T}_N$ denote the set of all strictly decreasing sequences of real numbers of length $N$, we see that the eigenvalues of $M_N$ induce a probability measure $\rho[M_N]$ on $\mathbb{T}_N$. Furthermore, we may define a probability measure $\tilde{\rho}[M_N]$ on $\mathbb{T}_N\times\mathbb{T}_{N-1}$ given by the eigenvalues of $M_N$ and its principal $(N-1)\times(N-1)$ submatrix.

The counting measures $m[\rho]$ and $d[\tilde{\rho}]$ can be defined in the same way as discussed previously. Then, one can produce an analogue of LLN-appropriateness for the sequence $(M_N)_{N \ge 1}$ such that the limiting measures of $m[\rho]$ and $d[\tilde{\rho}]$ are $\mathbf{m}$ and $\mathbf{d}$, respectively. Moreover, the measures $\mathbf{m}$ and $\mathbf{d}$ satisfy
\begin{equation}
\label{eq:MK}
\exp\left(\sum_{k=1}^\infty\frac{z^k}{k}\int_\mathbb{R}x^k \mathbf{d}(dx)\right)=\sum_{k=0}^\infty z^k\int_\mathbb{R}x^k\mathbf{m}(dx).
\end{equation}
This can be shown using similar methods as ours, replacing Schur functions with multivariate Bessel functions. For more details, see Section 2 of \cite{GS}. 
We note that this is nontrivial and do not discuss the proof further. Moreover, the unitarily invariant matrix ensemble can be recovered from the random signatures with a semiclassical limit. A partial description of this semiclassical limit is provided in \Cref{sec:bijection_proof}, where we prove results about the bijective nature of the correspondences.

Clearly, our statement is non-rigorous, and without proof. However, we note that \eqref{eq:MK} was shown to hold for the Wigner and Wishart ensembles in \cite{Bu}, which intersects the above result in the case of the GUE ensemble.

\subsection{Quantized Markov-Krein Correspondence Bijection} \label{ssec:QMK_correspondence}

The next main result establishes that the relation \eqref{eq:DMK} produces a bijection between two classes of objects. An analogous result was proved by Kerov in \cite{KEROV} for the relation \eqref{eq:MK}, which is the Markov-Krein correspondence. We begin by introducing some classes of objects.

Let $\cM[a,b]$ denote the set of probability measures supported on the interval $[a,b]$. For any $\mu\in\cM[a,b]$, define
\[\mu_k=\int_{-\infty}^{\infty} t^k d\mu(t).\]
Let $\ctM[a,b]$ denote the set of probability measures supported on the interval $[a,b]$ with density bounded between $0$ and $1$. A continual Young diagram is defined to be a function $w:\R\to\R$ that satisfies
\begin{itemize}
    \item $|w(x_1)-w(x_2)|\le|x_1-x_2|$ for all $x_1,x_2\in\R$.
    \item There exists $x_0\in\R$ such that $w(x)=|x-x_0|$ for sufficiently large $x_0$.
\end{itemize}
For an interval $[a,b]$, let $\cD[a,b]$ denote the set of continual Young diagrams satisfying $w(x)=|x-x_0|$ for all $x\not\in[a,b]$. Furthermore, let $\ctD[a,b]$ denote the set of $w\in\cD[a,b]$ such that $R_w(u)>-1$ for real $u$ outside the interval $[a,b]$ ($R_w$ is defined below). 

Define the function $p_k:\cD[a,b]\to\R$ for $k\in\Z_{\ge 1}$ by
\begin{equation}
\label{eq:pkdef}
p_k(w)=\frac{1}{2}\int_a^b t^k w''(t)dt.
\end{equation}
Note that $w''$ is $0$ outside of the interval $[a,b]$.

For a measure $\mu\in\cM[a,b]$, define its $R$-function to be
\[R_\mu(u)=\int_a^b\frac{d\mu(t)}{u-t} = \frac{1}{u}\sum_{k=0}^\infty u^{-k}\mu_k.\]
Similarly, for a measure $\psi\in\ctM[a,b]$, define its $R$-function to be
\[R_\psi(u)=-1+\exp\int_a^b\frac{d\psi(t)}{u-t} = -1+\exp\left(\frac{1}{u}\sum_{k=0}^\infty u^{-k}\psi_k\right).\]
Also, for a diagram $w\in\cD[a,b]$ with associated $\sigma(x)=\frac{1}{2}(w(x)-|x|)$, define its $R$-function to be
\[R_w(u)=\frac{1}{u}\exp\int_a^b\frac{d\sigma(t)}{t-u} = \frac{1}{u}\exp\sum_{k=1}^\infty\frac{u^{-k}}{k}p_k(w).\]
These functions are holomorphic outside the interval $[a,b]$.

As our next main result, we will prove that \eqref{eq:DMK} produces a bijective correspondence between $\ctM[a,b]$ and $\ctD[a,b]$.

\begin{theorem}
\label{thm:BQMK}
There is a bijective correspondence between $\ctM[a,b]$ and $\ctD[a,b]$ such that $\psi\leftrightarrow w$ if and only if
\[\frac{1}{z}\left(-1+\exp\left(z\sum_{k=0}^\infty \psi_k z^k\right)\right)=\exp\left(\sum_{k=1}^\infty\frac{p_k(w)}{k}z^k\right).\]
The above relation is equivalent to $R_\psi(u)=R_w(u)$.
\end{theorem}

Moreover, the Markov-Krein correspondence stated in \eqref{eq:MK}, shown by Krein and Nudelman in \cite{KREIN}, is bijective as well. The statement below, due to Kerov, is that the Markov-Krein correspondence is a bijection between $\cM[a,b]$ and $\cD[a,b]$.

\begin{theorem}[\cite{KEROV}*{p.\,107}]
\label{thm:BMK}
There is a bijective correspondence between $\cM[a,b]$ and $\cD[a,b]$ such that $\mu\leftrightarrow w$ if and only if
\[\sum_{k=0}^\infty\mu_k z^k = \exp\left(\sum_{k=1}^\infty\frac{p_k(w)}{k}z^k\right).\]
The above relation is equivalent to $R_\mu(u)=R_w(u)$.
\end{theorem}

One can realize \Cref{thm:BMK} as a semiclassical limit of \Cref{thm:BQMK}. We give a heuristic description of this semiclassical limit in \Cref{sec:bijection_proof}.

\subsection{Connection between Theorems \ref{thm:DMK} and \ref{thm:BQMK}}
The measure $\mathbf{m}$ lives in the space of probability measures with compact support and density bounded by $1$. Instead of looking at $\mathbf{d}$, we will be looking at a related object viewable as a second anti-derivative of $\mathbf{d}$.

Let $\{x_i\}$ and $\{y_i\}$ be two \textit{interlacing} sequences of real numbers, i.e.
\[x_1\le y_1\le x_2\le\cdots\le x_{N-1}\le y_{N-1}\le x_N.\]
Define $w^{\{x_i\},\{y_i\}}(x)$ to be the \textit{rectangular Young diagram} of $\{x_i\}$ and $\{y_i\}$ in the following way. Let $z_0=\sum_{i=1}^N x_i-\sum_{i=1}^{N-1}y_i$. Then, $w^{\{x_i\},\{y_i\}}(x)$ is the unique continuous function with the following properties.
\begin{itemize}
    \item $w^{\{x_i\},\{y_i\}}(x)=|x-z_0|$ for $x\le x_1$ and $x\ge x_N$.
    \item $\frac{d}{dx}w^{\{x_i\},\{y_i\}}(x) = 1$ for $x_i< x < y_i$ and $\frac{d}{dx}w^{\{x_i\},\{y_i\}}(x)=-1$ for $y_i<x<x_{i+1}$.
\end{itemize}
An example of a rectangular Young diagram is in Figure \ref{fig:interlacing}. For a probability measure $\rh$ on $\GT_N$, $w[\rh]$ is defined to be the random rectangular Young diagram $w^{\{x_i\},\{y_i\}}(x)$ with
\[\{x_i\}=\left\{\frac{\lambda_i+N-i}{N}\right\},  \quad \quad \{y_i\}=\left\{\frac{\mu_i+N-1-i}{N}\right\}\]
where $\lambda$ is distributed according to $\rh$ and $\mu$ is distributed according to $\pi_{N,N-1}\lambda$. For an LLN-appropriate sequence of probability measures $\{\rho_N\}$, convergence in probability of $d[\rho_N]$ to some measure $\mathbf{d}$ implies convergence of $w[\rho_N]$ to a limiting continual Young diagram $w$ such that $\frac{1}{2}w''(x)$ is given by the density of $\mathbf{d}$. This follows from the lemma below.
\begin{figure}[ht]
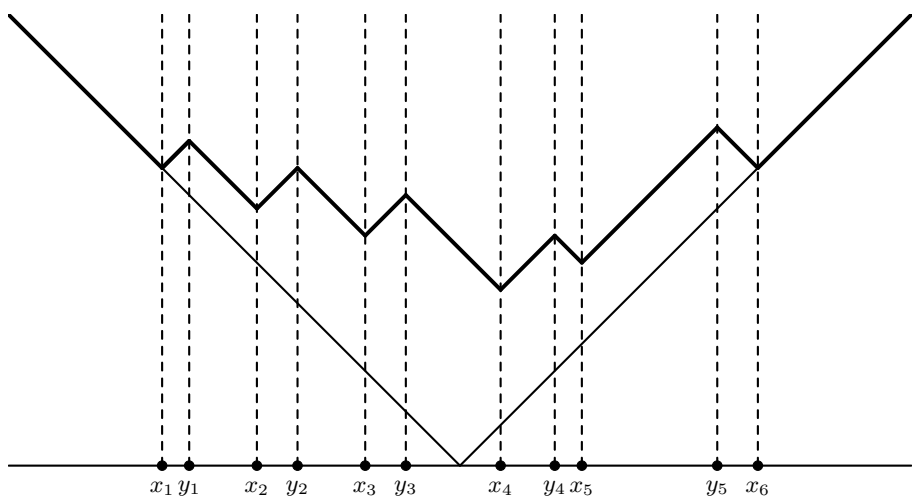

\subfile{Interlacing}
\caption{Interlacing sequences and rectangular Young diagram with $N= 6$}
\label{fig:interlacing}
\end{figure}
\begin{lemma}[\cite{Bu}*{Lemma 2.1}]
\label{lem:diagram_top}
Let $\mathcal{F}[a,b]$ be the set of all Lipchitz-$1$ real valued function $f(x)$ supported on $[a,b]$. The weak topology on this set defined by the functionals
\[f(x)\to\int_{a}^b f(x)x^k dx\]
for $k\ge 0$ coincides with the uniform topology.
\end{lemma}
Now, \Cref{thm:DMK} and \eqref{eq:pkdef} imply that the limits of $m[\rho_N]$ and $w[\rho_N]$ are paired by quantized Markov-Krein bijection as in Theorem \ref{thm:BQMK}.

\section{Moments of Limiting Measure}
\label{sec:momlim}
Here we will explicitly compute the moments referenced in \Cref{thm:DMK}. This is important for our eventual proof of the theorem.
\begin{theorem}[\cite{BuG1}*{Theorem 5.1}]
\label{thm:momreg}
Suppose that a sequence of probability measures $\rho=\{\rho_N\}_{N\ge 1}$, where $\rho_N$ is a probability measure on $\GT_N$, is LLN-appropriate. Let $H(x)=H_\rho(x)$ be the associated function from Definition \ref{def:LLNA}. The random measure $m[\rho_N]$ converges as $N\to\infty$ in probability, in the sense of moments to a deterministic measure $\mathbf{m}$ on $\mathbb{R}$ with moments
\[\int_\R x^k\mathbf{m}(dx)=\left.\sum_{\ell=0}^k\frac{1}{(\ell+1)!}\binom{k}{\ell}\left(\frac{\del}{\del z}\right)^\ell\left(z^k H'(z)^{k-\ell}\right)\right|_{z=1}.\]
\end{theorem}
\begin{theorem}
\label{thm:momdiff}
Suppose that a sequence of probability measures $\rho=\{\rho_N\}_{N\ge 1}$, where $\rho_N$ is a probability measure on $\GT_N$, is LLN-appropriate. Let $H(x)=H_\rho(x)$ be the associated function from Definition \ref{def:LLNA}. The random measure $d[\rho_N]$ converges as $N\to\infty$ in probability, in the sense of moments to a deterministic measure $\mathbf{d}$ on $\mathbb{R}$ with moments
\[\int_\R x^k\mathbf{d}(dx)=\left.\sum_{\ell=0}^k\frac{1}{\ell!}\binom{k}{\ell}\left(\frac{\del}{\del z}\right)^\ell\left(z^k H'(z)^{k-\ell}\right)\right|_{z=1}.\]
\end{theorem}

\subsection{Moments and Schur Generating Functions}
\label{sec:moments and schur gf}
In this section, we will be introducing tools necessary to prove \Cref{thm:momreg,thm:momdiff}. The point of this section is that the moments of $m[\rh]$ and $d[\rh]$ can be found by applying certain differential operators on the Schur generating function of $\rh$.
\begin{definition}
\label{def:operator}
Define a differential operator on functions of $(x_1,\ldots,x_N)$ by
\[\cD_{N,k} := \frac{1}{\displaystyle\prod_{1\le i<j\le N}(x_i-x_j)}\left(\sum_{i=1}^N(x_i\del_i)^k\right)\prod_{1\le i<j\le N}(x_i-x_j)\]
where $\del_i := \frac{\del}{\del x_i}$.
\end{definition}
The following theorem shows how this can be used to find the moments of $m[\rh]$ and $d[\rh]$. Note that this is similar to \cite{BuG1}*{Theorem 4.5}.
\begin{theorem}
\label{thm:opmom}
Let $\rh$ be a probability distribution on $\GT_N$. Then, we have
\[\E\left[\left(\int_\mathbb{R}x^km[\rh](dx)\right)^n\right]=\left.\frac{1}{N^{n(k+1)}}\cD_{N,k}^nS_\rh(x_1,\ldots,x_N)\right|_{x_1=\cdots=x_N=1}\]
and
\[\E\left[\left(\int_\mathbb{R}x^kd[\rh](dx)\right)^n\right]=\left.\frac{1}{N^{nk}}\left(\sum_{\ell=0}^n\binom{n}{\ell}(-1)^{n-\ell}\cD_{N-1,k}^{n-\ell}\cD_{N,k}^\ell\right)S_\rh(x_1,\ldots,x_N)\right|_{x_1=\cdots=x_N=1}.\]
\end{theorem}
To prove this, we need the following lemma that allows us to get a handle on the Schur generating function after applying the projection map.
\begin{lemma}
\label{lem:ESGF}
For any $\lambda\in\GT_N$, we have
\[S_{\pi_{N,N-1}\delta_\lambda}(x_1,\ldots,x_{N-1})=\frac{s_\lambda(x_1,\ldots,x_{N-1},1)}{s_\lambda(1^N)}.\]
\end{lemma}
\begin{proof}
Let $V_N^\lambda$ be the irreducible representation of $\rmU(N)$ indexed by $\lambda$. Let it's restriction onto $\rmU(N-1)$ decompose into irreducible representations as
\[V_N^\lambda|_{\rmU(N-1)} = \bigoplus_{\mu \in \GT_{N-1}} m_\mu V_{N-1}^\mu. \]
Then,
\begin{align*}
    S_{\pi_{N,N-1}\delta_\lambda}(x_1,\ldots,x_{N-1}) &= \sum_{\mu\in\GT_{N-1}}\frac{m_\mu\dim(V_{N-1}^\mu)}{\dim (V_N^\lambda|_{\rmU(N-1)})}\frac{s_\mu(x_1,\ldots,x_{N-1})}{s_\mu(1^{N-1})} \\
    &= \sum_{\mu\in\GT_{N-1}}\frac{m_\mu}{\dim (V_N^\lambda|_{\rmU(N-1)})}s_\mu(x_1,\ldots,x_{N-1}).
\end{align*}
This is just the normalized character of $V_N^\lambda|_{\rmU(N-1)}$ due to the fact that characters add under direct sums. Therefore, the proof is complete, since the normalized character is also
\[\frac{s_\lambda(x_1,\ldots,x_{N-1},1)}{s_\lambda(1^N)}.\]
\end{proof}
We are now in position to prove Theorem \ref{thm:opmom}.
\begin{proof}[Proof of Theorem \ref{thm:opmom}]
The key to this theorem is that Schur functions $s_\lambda$ are eigenfunctions of $\cD_{N,k}$. In particular, for any $\lambda\in\GT_N$, we have
\[\cD_{N,k}s_\lambda(x_1,\ldots,x_N) = \left(\sum_{i=1}^N(\lambda_i+N-i)^k\right)s_\lambda(x_1,\ldots,x_N).\]
This is shown by a straightforward calculation with the determinant definition of Schur functions. We have that
\[\E\left[\left(\int_\R x^km[\rh](dx)\right)^n\right]=\sum_{\lambda\in\GT_N}\rh(\lambda)\left(\frac{1}{N}\sum_{i=1}^N\left(\frac{\lambda_i+N-i}{N}\right)^k\right)^n.\]
By the eigenfunction observation, this can be written as
\[\E\left[\left(\int_\R x^km[\rh](dx)\right)^n\right]=\frac{1}{N^{n(k+1)}}\sum_{\lambda\in\GT_N}\rh(\lambda)\frac{\cD_{N,k}^n s_\lambda(x_1,\ldots,x_N)}{s_\lambda(x_1,\ldots,x_N)}.\]
Evaluating the right side at $x_1=\cdots=x_N=1$, the first result is immediate by definition of $S_\rh$.

We will now show the second result. By similar logic as above, we may set
\begin{align*}
    T(x_1,\ldots,x_N)&:=\frac{1}{N^{\ell k}}\cD_{N,k}^\ell\sum_{\lambda\in\GT_N}\rh(\lambda)\frac{s_\lambda(x_1,\ldots,x_N)}{s_\lambda(1^N)}\\
    &= \sum_{\lambda\in\GT_N}\rh(\lambda)\left(\sum_{i=1}^N\left(\frac{\lambda_i+N-i}{N}\right)^k\right)^\ell\frac{s_\lambda(x_1,\ldots,x_N)}{s_\lambda(1^N)}.
\end{align*}
Our goal is to evaluate
\[\left.\frac{1}{N^{(n-\ell)k}}\cD_{N-1,k}^{n-\ell}T(x_1,\ldots,x_N)\right|_{x_i=1}.\]
Since $\cD_{N-1,k}$ doesn't act on the variable $x_N$, we may evaluate the above expression with $T(x_1,\ldots,x_N)$ replaced by $T(x_1,\ldots,x_{N-1},1)$. By Lemma \ref{lem:ESGF}, 
\begin{align*}
&\frac{1}{N^{(n-\ell)k}}\cD_{N-1,k}^{n-\ell}\frac{s_\lambda(x_1,\ldots,x_{N-1},1)}{s_\lambda(1^N)}\\
=&\frac{1}{N^{(n-\ell)k}}\cD_{N-1,k}^{n-\ell}\sum_{\mu\in\GT_{N-1}}\pi_{N,N-1}\delta_\lambda(\mu)\frac{s_\mu(x_1,\ldots,x_{N-1})}{s_\mu(1^{N-1})} \\
=&\sum_{\mu\in\GT_{N-1}}\pi_{N,N-1}\delta_\lambda(\mu)\left(\sum_{i=1}^{N-1}\left(\frac{\mu_i+N-1-i}{N}\right)^k\right)^{n-\ell}\frac{s_\mu(x_1,\ldots,x_{N-1})}{s_\mu(1^{N-1})}, \end{align*}
so
\begin{align*}
&\left.\frac{1}{N^{(n-\ell)k}}\cD_{N-1,k}^{n-\ell}T(x_1,\ldots,x_{N-1},1)\right|_{x_i=1} \\
=& \sum_{\lambda\in\GT_N}\rh(\lambda)\left(\sum_{i=1}^N\left(\frac{\lambda_i+N-i}{N}\right)^k\right)^\ell\sum_{\mu\in\GT_{N-1}}\pi_{N,N-1}\delta_\lambda(\mu)\left(\sum_{i=1}^{N-1}\left(\frac{\mu_i+N-1-i}{N}\right)^k\right)^{n-\ell} \\
=&\sum_{\substack{\lambda\in\GT_N \\ \mu\in\GT_{N-1}}}\rh(\lambda)\pi_{N,N-1}\delta_\lambda(\mu)\left(\sum_{i=1}^N\left(\frac{\lambda_i+N-i}{N}\right)^k\right)^\ell\left(\sum_{i=1}^{N-1}\left(\frac{\mu_i+N-1-i}{N}\right)^k\right)^{n-\ell}.
\end{align*}
The theorem now follows by the binomial theorem.
\end{proof}

\subsection{Two Lemmas}
We introduce two lemmas that will be useful in our analysis of the moments of the counting measure. In particular, they help to reduce the number of variables in the calculation. Before proceeding, we introduce the following notation.

Given a function $f(z_1,\ldots,z_n)$, define
\[\cyc f(z_1,\ldots,z_n):=f(z_1,\ldots,z_n)+f(z_2,\ldots,z_n,z_1)+\cdots+f(z_n,z_1,\ldots,z_{n-1}).\]
\begin{lemma}[\cite{BuG1}*{Lemma 5.5}]
\label{lem:BG 5.5}
Let $g(z)$ be a function analytic in some neighborhood of $z=1$. Then
\[\lim_{z_i\to 1}\cyc \frac{g(z_1)}{(z_1-z_2)(z_1-z_3)\cdots(z_1-z_n)} = \left.\frac{1}{(n-1)!}\frac{\del^{n-1}g(z)}{\del z^{n-1}}\right|_{z=1}.\]
\end{lemma}
\begin{lemma}
\label{lem:expansion}
For positive integers $k$, we have that
\[\cD_{N,k} = \sum_{m=1}^k\stirling{k}{m}\sum_{\ell=0}^m\binom{m}{\ell}\ell!\sum_{\{i_0,\ldots,i_\ell\}\subseteq[N]}\cyc \frac{x_{i_0}^m\del_{i_0}^{m-\ell}}{(x_{i_0}-x_{i_1})\cdots(x_{i_0}-x_{i_\ell})},\]
where $\stirling{k}{m}$ are Stirling numbers of the second kind.
\end{lemma}
\begin{proof}
Let $\Delta(x)=\prod_{i<j}(x_i-x_j)$ denote the Vandermonde determinant. We first compute $\del_p^m\Delta(x)$. By the Leibniz rule, we have
\[\del_p^m\Delta(x) = \sum_{\substack{(k_{ij})_{1\le i<j\le N} \\k_{i,j}\in\mathbb{Z}_{\ge 0} \\ \sum k_{i,j}=m}}\binom{m}{k_{1,2},\ldots,k_{N-1,N}}\prod_{i<j}\del_p^{k_{i,j}}(x_i-x_j).\]
With some work, this reduces to
\[\del_p^m\Delta(x) = m!\prod_{i<j}(x_i-x_j)\sum_{\substack{S\subseteq[N]\setminus\{p\}\\|S|=m}}\frac{1}{\displaystyle\prod_{i\in S}(x_p-x_i)}.\]
It is well known that
\[\sum_{i=1}^N (x_i \partial_i)^k = \sum_{i=1}^N\sum_{m=1}^k\stirling{k}{m}x_i^m\del_i^m,\]
so in fact
\begin{align*}
\cD_{N,k} &= \Delta(x)^{-1}\sum_{m=1}^k\stirling{k}{m}\sum_{i=1}^N x_i^m\sum_{\ell=0}^m\binom{m}{\ell}(\del_i^\ell\Delta(x))\del_i^{m-\ell} \\
&= \sum_{m=1}^k\stirling{k}{m}\sum_{\ell=0}^m\binom{m}{\ell}\ell!\sum_{i=1}^N x_i^m\sum_{\substack{S\subseteq[N]\setminus\{i\}\\|S|=\ell}}\left(\prod_{j\in S}\frac{1}{x_i-x_j}\right)\del_i^{m-\ell}.
\end{align*}
But we have that
\[\sum_{i=1}^N x_i^m\sum_{\substack{S\subseteq[N]\setminus\{i\}\\|S|=\ell}}\left(\prod_{j\in S}\frac{1}{x_i-x_j}\right)\del_i^{m-\ell} = \sum_{\{i_0,\ldots,i_\ell\}\subseteq[N]}\cyc \frac{x_{i_0}^m\del_{i_0}^{m-\ell}}{(x_{i_0}-x_{i_1})\cdots(x_{i_0}-x_{i_\ell})},\]
which completes the proof.
\end{proof}

\subsection{Proof of Theorem \ref{thm:momreg}}
Note that the proof presented here is very similar to that given in \cite{BuG1}*{Section 5.2}. It is necessary to show
\begin{align}\label{eq:firstorder}
\lim_{N\to\infty}\E\left(\int_\R x^k m[\rho_N](dx)\right)=\sum_{\ell=0}^k\frac{1}{(\ell+1)!}\binom{k}{\ell}\left.\left(\frac{\del}{\del z}\right)^\ell\left(z^k H'(z)^{k-\ell}\right)\right|_{z=1}
\end{align}
and
\begin{align}\label{eq:secondorder}\lim_{N\to\infty}\E\left[\left(\int_\R x^k m[\rho_N](dx)\right)^2\right] = \lim_{N\to\infty}\left[\E\left(\int_\R x^k m[\rho_N](dx)\right)\right]^2.\end{align}
First, we will show (\ref{eq:firstorder}). By Theorem \ref{thm:opmom}, 
\begin{align}\label{eq:4}
\E\left(\int_\R x^k m[\rho_N](dx)\right) = \frac{1}{N^{k+1}}\lim_{x_i\to 1}\cD_{N,k}S_{\rho_N}(x_1,\ldots,x_N).
\end{align}
Define $T_N(x_1,\ldots,x_N)$ by 
\begin{align}\label{eq:almost_mult}
S_{\rho_N}(x_1,\ldots,x_N)=\exp\left(\sum_{i=1}^N NH(x_i)\right)T_N(x_1,\ldots,x_N),
\end{align}
where $T_N$ is analytic in some open neighborhood of $(1^N)$. By Definition \ref{def:LLNA}, $T_N(1^N)=1$ and, for any fixed $k$,
\[\lim_{N\to\infty}\frac{1}{N}\log T_N(x_1,\ldots,x_k,1^{N-k})=0\] uniformly in some open neighborhood of $(1^k)$. Differentiating the above result, we have
\begin{align}\label{eq:deriv_die}
\lim_{N\to\infty}\frac{1}{N}\frac{\del_1^{a_1}\cdots\del_k^{a_k} T_N(x_1,\ldots,x_k,1^{N-k})}{T_N(x_1,\ldots,x_k,1^{N-k})}=0
\end{align}
for any $(a_1,\ldots,a_k)\in\Z_{\ge 0}^k$. By Lemma \ref{lem:expansion}, $\cD_{N,k}S_{\rho_N}$ is
\begin{align}\label{eq:7}
\cD_{N,k}S_{\rho_N} = \sum_{m=1}^k\stirling{k}{m}\sum_{\ell=0}^m\binom{m}{\ell}\ell!\sum_{\{i_0,\ldots,i_\ell\}\subseteq[N]}\cyc \frac{x_{i_0}^m\del_{i_0}^{m-\ell}S_{\rho_N}}{(x_{i_0}-x_{i_1})\cdots(x_{i_0}-x_{i_\ell})}.
\end{align}
We wish to take the limit $x_i\to 1$ of both sides. To do this, consider the following proposition.
\begin{proposition}
\label{prelimit}
The leading order term of
\[L:=\lim_{x_i\to 1}\cyc \frac{x_{i_0}^m\del_{i_0}^{m-\ell}S_{\rho_N}}{(x_{i_0}-x_{i_1})\cdots(x_{i_0}-x_{i_\ell})}\]
is
\[\frac{1}{\ell!}N^{m-\ell}\left.\left(\frac{\del}{\del z}\right)^\ell(z^mH'(z)^{m-\ell})\right|_{z=1}.\]
\end{proposition}
\begin{proof}
Since all functions are symmetric, we may assume $\{i_0,\ldots,i_\ell\}=\{1,\ldots,\ell+1\}$. For $r\ge \ell+2$, we can set $x_r=1$, so
\[L=\lim_{x_1,\ldots,x_{\ell+1}\to 1}\cyc\frac{x_1^m\del_1^{m-\ell}S_{\rho_N}(x_1,\ldots,x_{\ell+1},1^{N-\ell-1})}{(x_{1}-x_{2})\cdots(x_{1}-x_{\ell+1})}.\]
From (\ref{eq:almost_mult}) and (\ref{eq:deriv_die}), the leading order term of
\[\del_1^{m-\ell}S_{\rho_N}(x_1,\ldots,x_{\ell+1},1^{N-\ell-1})\]
is $N^{m-\ell}H'(x_1)^{m-\ell}S_{\rho_N}(x_1,\ldots,x_{\ell+1},1^{N-\ell-1})$. Thus, the leading order term of $L$ is
\[N^{m-\ell}\lim_{x_1,\ldots,x_{\ell+1}\to 1}\cyc\frac{x_1^mH'(x_1)^{m-\ell}}{(x_{1}-x_{2})\cdots(x_{1}-x_{\ell+1})}S_{\rho_N}(x_1,\ldots,x_{\ell+1},1^{N-\ell-1}).\]
Applying Lemma \ref{lem:BG 5.5} and noting that $S_{\rho_N}(1^N)=1$ yields the desired result.
\end{proof}
Therefore, the leading order term of the right side of (\ref{eq:7}) under the limit $x_i\to 1$ is
\begin{align}\label{eq:setcount}
\left.\sum_{m=1}^k\stirling{k}{m}\sum_{\ell=0}^m\binom{m}{\ell}\binom{N}{\ell+1}N^{m-\ell}\left(\frac{\del}{\del z}\right)^\ell(z^mH'(z)^{m-\ell})\right|_{z=1}.
\end{align}
Since the order of the summand is $N^{\ell+1}N^{m-\ell}=N^{m+1}$, the only contribution in the limit comes from $m=k$, so the leading order term is
\[\left.N^{k+1}\sum_{\ell=0}^k\binom{k}{\ell}\frac{1}{(\ell+1)!}\left(\frac{\del}{\del z}\right)^\ell(z^mH'(z)^{m-\ell})\right|_{z=1}.\]
Thus, by (\ref{eq:4}), 
\[\lim_{N\to\infty}\E\left(\int_\R x^k m[\rho_N](dx)\right)=\sum_{\ell=0}^k\frac{1}{(\ell+1)!}\binom{k}{\ell}\left.\left(\frac{\del}{\del z}\right)^\ell\left(z^k H'(z)^{k-\ell}\right)\right|_{z=1},\]
as desired.

We will now show (\ref{eq:secondorder}). It suffices to show that the leading order terms of $(\cD_{N,k}S_{\rho_N})^2$ and $(\cD_{N,k})^2S_{\rho_N}$ match up when we take $x_i\to 1$. We see that
\begin{align*}
(\cD_{N,k}S_{\rho_N})^2 =& \sum_{m=1}^k\sum_{m'=1}^k\stirling{k}{m}\stirling{k}{m'}\sum_{\ell=0}^m\sum_{\ell'=0}^{m'}\binom{m}{\ell}\ell!\binom{m'}{\ell'}\ell'!\\
&\sum_{\{i_0,\ldots,i_\ell\}\subseteq[N]}\sum_{\{i_0',\ldots,i_{\ell'}'\}\subseteq[N]}\cyc\cyc \frac{(x_{i_0}^m\del_{i_0}^{m-\ell}S_{\rho_N})(x_{i_0'}^{m'}\del_{i_0'}^{m'-\ell'}S_{\rho_N})}{(x_{i_0}-x_{i_1})\cdots(x_{i_0}-x_{i_\ell})(x_{i_0'}-x_{i_1'})\cdots(x_{i_0'}-x_{i_\ell'})}
\end{align*}
and
\begin{align*}
\numberthis \label{eq:Dnk2}(\cD_{N,k})^2S_{\rho_N} =& \sum_{m=1}^k\sum_{m'=1}^k\stirling{k}{m}\stirling{k}{m'}\sum_{\ell=0}^m\sum_{\ell'=0}^{m'}\binom{m}{\ell}\ell!\binom{m'}{\ell'}\ell'!\\
&\sum_{\{i_0,\ldots,i_\ell\}\subseteq[N]}\sum_{\{i_0',\ldots,i_{\ell'}'\}\subseteq[N]}\cyc\cyc \frac{(x_{i_0}^m\del_{i_0}^{m-\ell}x_{i_0'}^{m'}\del_{i_0'}^{m'-\ell'})S_{\rho_N}}{(x_{i_0}-x_{i_1})\cdots(x_{i_0}-x_{i_\ell})(x_{i_0'}-x_{i_1'})\cdots(x_{i_0'}-x_{i_\ell'})}.
\end{align*}
The only difference in the two is the numerator of the final summand. We have the following claim that is an analogue of Proposition \ref{prelimit}.
\begin{proposition}
\label{prop:equal_leading_terms}
The leading order terms of
\[\lim_{x_i\to 1}\cyc\cyc \frac{(x_{i_0}^m\del_{i_0}^{m-\ell}S_{\rho_N})(x_{i_0'}^{m'}\del_{i_0'}^{m'-\ell'}S_{\rho_N})}{(x_{i_0}-x_{i_1})\cdots(x_{i_0}-x_{i_\ell})(x_{i_0'}-x_{i_1'})\cdots(x_{i_0'}-x_{i_\ell'})}\]
and
\[\lim_{x_i\to 1}\cyc\cyc \frac{(x_{i_0}^m\del_{i_0}^{m-\ell}x_{i_0'}^{m'}\del_{i_0'}^{m'-\ell'})S_{\rho_N}}{(x_{i_0}-x_{i_1})\cdots(x_{i_0}-x_{i_\ell})(x_{i_0'}-x_{i_1'})\cdots(x_{i_0'}-x_{i_\ell'})}\]
are identical.
\end{proposition}

\begin{proof}
As in the proof of Proposition \ref{prelimit}, set $x_r=1$ for all $r\not\in\{i_0,\ldots,i_\ell\}\cup\{i_0',\ldots,i_\ell'\}$. Then, using (\ref{eq:almost_mult}) and (\ref{eq:deriv_die}), the leading order term of $(x_{i_0}^m\del_{i_0}^{m-\ell}S_{\rho_N})(x_{i_0'}^{m'}\del_{i_0'}^{m'-\ell'}S_{\rho_N})$ is
\[N^{m-\ell+m'-\ell'}x_{i_0}^mx_{i_0'}^{m'}H'(x_{i_0})^{m-\ell}H'(x_{i_0'})^{m'-\ell'}S_{\rho_N}^2.\]
Now, the leading order term of $(x_{i_0}^m\del_{i_0}^{m-\ell}x_{i_0'}^{m'}\del_{i_0'}^{m'-\ell'})S_{\rho_N}$ is the same as the leading order term of
\[N^{m'-\ell'}x_{i_0}^m\del_{i_0}^{m-\ell}(x_{i_0'}^{m'}H'(x_{i_0'})^{m'-l'}S_{\rho_N}).\]
When applying the product rule, there is a new factor of $N$ after applying the derivative on the $S_{\rho_N}$, and none for any of the other terms. Therefore, the leading order of this is
\[N^{m-\ell+m'-\ell'}x_{i_0}^mx_{i_0'}^{m'}H'(x_{i_0})^{m-\ell}H'(x_{i_0'})^{m'-\ell'}S_{\rho_N}.\]
Taking $x_i\to 1$ and using Lemma \ref{lem:BG 5.5}, we will get the same thing since $S_{\rho_N}(1^N)^2=S_{\rho_N}(1^N)=1$. Thus, the leading order terms match as desired.
\end{proof}
This shows (\ref{eq:secondorder}), and the proof of Theorem \ref{thm:momreg} is complete.

\subsection{Proof of Theorem \ref{thm:momdiff}}

This proof is almost identical to the proof of Theorem \ref{thm:momreg}, except for a few key differences. We will highlight those differences, and the rest of the proof carries through in a similar way.

As in the previous proof, it is required to show
\begin{align}\label{eq:firstorder_diff}
\lim_{N\to\infty}\E\left(\int_\R x^k d[\rho_N](dx)\right)=\sum_{\ell=0}^k\frac{1}{\ell!}\binom{k}{\ell}\left.\left(\frac{\del}{\del z}\right)^\ell\left(z^k H'(z)^{k-\ell}\right)\right|_{z=1}
\end{align}
and
\begin{align}\label{eq:secondorder_diff}\lim_{N\to\infty}\E\left[\left(\int_\R x^k d[\rho_N](dx)\right)^2\right] = \lim_{N\to\infty}\left(\E\left(\int_\R x^k d[\rho_N](dx)\right)\right)^2.\end{align}
First, we will show (\ref{eq:firstorder_diff}). From Theorem \ref{thm:opmom}, 
\begin{align}\label{eq:4_diff}
\E\left(\int_\R x^k d[\rho_N](dx)\right) = \frac{1}{N^k}\lim_{x_i\to 1}(\cD_{N,k}-\cD_{N-1,k})S_{\rho_N}(x_1,\ldots,x_N),
\end{align}
and from Lemma \ref{lem:expansion},
\[\cD_{N,k}-\cD_{N-1,k} = \sum_{m=1}^k\stirling{k}{m}\sum_{\ell=0}^m\binom{m}{\ell}\ell!\sum_{\substack{\{i_0,\ldots,i_\ell\}\subseteq[N]\\N\in\{i_0,\ldots,i_\ell\}}}\cyc \frac{x_{i_0}^m\del_{i_0}^{m-\ell}}{(x_{i_0}-x_{i_1})\cdots(x_{i_0}-x_{i_\ell})}.\]
Thus, the proof of (\ref{eq:firstorder_diff}) is the same as the proof of (\ref{eq:firstorder}), except the step at (\ref{eq:setcount}). Now, instead of $\binom{N}{\ell+1}$, we will have $\binom{N-1}{\ell}$. The loss of order of $N$ by $1$ here is compensated by the same loss in (\ref{eq:4_diff}). However, the factor of $\frac{1}{(\ell+1)!}$ in Theorem \ref{thm:momreg} that came from $\binom{N}{\ell+1}$ becomes $\frac{1}{\ell!}$. Thus, the moments of $d[\rho_N]$ converge to
\[\left.\sum_{\ell=0}^k\frac{1}{\ell!}\binom{k}{\ell}\left(\frac{\del}{\del z}\right)^\ell\left(z^k H'(z)^{k-\ell}\right)\right|_{z=1},\]
as desired.

The modification for (\ref{eq:secondorder_diff}) is slightly more complicated, since $\cD_{N,K}^2$ in (\ref{eq:Dnk2}) is replaced with
\[\cD_{N,k}^2-2\cD_{N-1,k}\cD_{N,k}+\cD_{N-1,k}^2=(\cD_{N,k}-\cD_{N-1,k})^2+(\cD_{N,k}\cD_{N-1,k}-\cD_{N-1,k}\cD_{N,k})\]
instead of just $(\cD_{N,k}-\cD_{N-1,k})^2$. Ignoring the commutator term, we see that (\ref{eq:secondorder_diff}) holds using a similar proof as in Theorem \ref{thm:momreg}, with
\[\frac{1}{N^{2k}}(\cD_{N,k}-\cD_{N-1,k})^2S_{\rho_N}=\frac{1}{N^{2k}}((\cD_{N,k}-\cD_{N-1,k})S_{\rho_N})^2\]
in the limit $N\to\infty$. So, it suffices to show that the commutator term doesn't contribute to the leading order. In other words, it suffices to show that
\begin{align}
\label{eq:lim_commutator}
\lim_{N\to\infty}\frac{1}{N^{2k}}\lim_{x_i\to 1}(\cD_{N,k}\cD_{N-1,k}-\cD_{N-1,k}\cD_{N,k})S_{\rho_N}(x_1,\ldots,x_N)=0.
\end{align}
We see that
\begin{align*}
\numberthis \label{eq:commutator_explicit}(\cD_{N,k}\cD_{N-1,k}-\cD_{N-1,k}\cD_{N,k})S_{\rho_N} &= \sum_{m=1}^k\sum_{m'=1}^k\stirling{k}{m}\stirling{k}{m'}\sum_{\ell=0}^m\sum_{\ell'=0}^{m'}\binom{m}{\ell}\ell!\binom{m'}{\ell'}\ell'!\\
&\left(\sum_{\{i_0,\ldots,i_\ell\}\subseteq[N]}\sum_{\{i_0',\ldots,i_{\ell'}'\}\subseteq[N-1]}-\sum_{\{i_0,\ldots,i_\ell\}\subseteq[N-1]}\sum_{\{i_0',\ldots,i_{\ell'}'\}\subseteq[N]}\right) \\
&\cyc\cyc \frac{(x_{i_0}^m\del_{i_0}^{m-\ell}x_{i_0'}^{m'}\del_{i_0'}^{m'-\ell'})S_{\rho_n}}{(x_{i_0}-x_{i_1})\cdots(x_{i_0}-x_{i_\ell})(x_{i_0'}-x_{i_1'})\cdots(x_{i_0'}-x_{i_\ell'})}.
\end{align*}
Note that if $i_0\ne i_0'$, the operators $x_{i_0}^m\del_{i_0}^{m-\ell}$ and $x_{i_0'}^{m'}\del_{i_0'}^{m'-\ell'}$ commute, so we may restrict our attention to sets $\{i_0,\ldots,i_\ell\}$ and $\{i_0',\ldots,i_{\ell'}'\}$ having nonempty intersection. Furthermore, if both sets $\{i_0,\ldots,i_\ell\}$ and $\{i_0',\ldots,i_{\ell'}'\}$ are contained in $[N-1]$, the two summations cancel. Thus, we may rewrite (\ref{eq:commutator_explicit}) as
\begin{align*}
(\cD_{N,k}\cD_{N-1,k}-\cD_{N-1,k} &  \cD_{N,k})S_{\rho_N} = \sum_{m=1}^k\sum_{m'=1}^k\stirling{k}{m}\stirling{k}{m'}\sum_{\ell=0}^m\sum_{\ell'=0}^{m'}\binom{m}{\ell}\ell!\binom{m'}{\ell'}\ell'!\\
&\sum_{\substack{\{i_0,\ldots,i_\ell\}\subseteq[N] \\ N\in \{i_0,\ldots,i_\ell\} \\ \{i_0',\ldots,i_{\ell'}'\}\subseteq[N-1] \\  \{i_0,\ldots,i_\ell\}\cap\{i_0',\ldots,i_{\ell'}'\}\ne\emptyset}}\cyc\cyc \frac{(x_{i_0}^m\del_{i_0}^{m-\ell}x_{i_0'}^{m'}\del_{i_0'}^{m'-\ell'}-x_{i_0'}^{m'}\del_{i_0'}^{m'-\ell'}x_{i_0}^m\del_{i_0}^{m-\ell})S_{\rho_N}}{(x_{i_0}-x_{i_1})\cdots(x_{i_0}-x_{i_\ell})(x_{i_0'}-x_{i_1'})\cdots(x_{i_0'}-x_{i_\ell'})}.
\end{align*}
We already noted that $(x_{i_0}^m\del_{i_0}^{m-\ell}x_{i_0'}^{m'}\del_{i_0'}^{m'-\ell'}-x_{i_0'}^{m'}\del_{i_0'}^{m'-\ell'}x_{i_0}^m\del_{i_0}^{m-\ell})S_{\rho_N}=0$ if $i_0\ne i_0'$, but even if $i_0=i_0'$, its order is bounded above by $N^{m+m'-\ell-\ell'-1}$ using essentially the same argument as in the proof of Proposition \ref{prop:equal_leading_terms}. The number of pairs of sets satisfying the four conditions in the summation is on the order of $N^{-1}\cdot N^{\ell}\cdot N^{\ell'+1}=N^{\ell+\ell'}$, so the order of
\[\sum_{\substack{\{i_0,\ldots,i_\ell\}\subseteq[N] \\ N\in \{i_0,\ldots,i_\ell\} \\ \{i_0',\ldots,i_{\ell'}'\}\subseteq[N-1] \\  \{i_0,\ldots,i_\ell\}\cap\{i_0',\ldots,i_{\ell'}'\}\ne\emptyset}}\cyc\cyc \frac{(x_{i_0}^m\del_{i_0}^{m-\ell}x_{i_0'}^{m'}\del_{i_0'}^{m'-\ell'}-x_{i_0'}^{m'}\del_{i_0'}^{m'-\ell'}x_{i_0}^m\del_{i_0}^{m-\ell})S_{\rho_N}}{(x_{i_0}-x_{i_1})\cdots(x_{i_0}-x_{i_\ell})(x_{i_0'}-x_{i_1'})\cdots(x_{i_0'}-x_{i_\ell'})}\]
is at most $N^{m+m'-1}$, which is at most $N^{2k-1}$. The sum
\[\sum_{m=1}^k\sum_{m'=1}^k\stirling{k}{m}\stirling{k}{m'}\sum_{\ell=0}^m\sum_{\ell'=0}^{m'}\binom{m}{\ell}\ell!\binom{m'}{\ell'}\ell'!\]
is fixed with respect to $N$, so the order of
\[(\cD_{N,k}\cD_{N-1,k}-\cD_{N-1,k}\cD_{N,k})S_{\rho_N}\]
is at most $N^{2k-1}$. This proves (\ref{eq:lim_commutator}), completing the proof of Theorem \ref{thm:momdiff}.

\section{Proof of Theorem \ref{thm:DMK}}
\label{sec:theproof}
Assume the hypotheses of \Cref{thm:DMK}, and as before, let $H(x)=H_\rho(x)$. For convenience of notation, define
\[m_k:=\int_\R x^k\mathbf{m}(dx)=\left.\sum_{\ell=0}^k\frac{1}{(\ell+1)!}\binom{k}{\ell}\left(\frac{\del}{\del z}\right)^\ell\left(z^k H'(z)^{k-\ell}\right)\right|_{z=1}\]
and
\[d_k := \int_\R x^k\mathbf{d}(dx)=\left.\sum_{\ell=0}^k\frac{1}{\ell!}\binom{k}{\ell}\left(\frac{\del}{\del z}\right)^\ell\left(z^k H'(z)^{k-\ell}\right)\right|_{z=1}. \]
It turns out these can be expressed as contour integrals, which will allow for more convenient computations in the proof of Theorem \ref{thm:DMK}.
\begin{proposition}
\label{prop: mcounter}
The moment $m_k$ can be expressed as the contour integral
\[m_k=\frac{1}{k+1}\oint_1\frac{1}{w}\left(wH'(w)+\frac{w}{w-1}\right)^{k+1}\frac{dw}{2\pi\i},\]
where the contour is counterclockwise around $1$. Similarly, the moment $d_k$ can be expressed as the contour integral
\[d_k = \oint_1\frac{1}{w-1}\left(wH'(w)+\frac{w}{w-1}\right)^k\frac{dw}{2\pi\i},\]
where the contour is counterclockwise around $1$.
\end{proposition}
\begin{proof}
Expanding with the binomial theorem yields
\begin{align*}
     \frac{1}{k+1}\oint_1\frac{1}{w}\left(wH'(w)+\frac{w}{w-1}\right)^{k+1}\frac{dw}{2\pi\i}
     = & \frac{1}{k+1}\left(\oint_1w^kH'(w)^{k+1} +\sum_{\ell=0}^k\dbinom{k+1}{\ell+1}\dfrac{w^{k}H'(w)^{k-\ell}}{(w-1)^{\ell+1}}\frac{dw}{2\pi\i}\right)
     \\ = & \sum_{\ell = 0}^k\frac{1}{\ell+1}\dbinom{k}{\ell}\oint_1 \dfrac{w^kH'(w)^{k-\ell}}{(w-1)^{\ell+1}}\frac{dw}{2\pi\i}.
\end{align*}
By Cauchy's Differentiation Formula,
\[\oint_1 \dfrac{w^kH'(w)^{k-\ell}}{(w-1)^{\ell+1}} \frac{dw}{2\pi\i} = \dfrac{\left.\left(\dfrac{\del}{\del z}\right)^\ell\left(z^k H'(z)^{k-\ell}\right)\right|_{z=1}}{\ell!},\]
so
\begin{align*}
     \frac{1}{k+1}\oint_1\frac{1}{w}\left(wH'(w)+\frac{w}{w-1}\right)^{k+1}\frac{dw}{2\pi\i} & =\left.\sum_{\ell=0}^k\frac{1}{(\ell+1)!}\binom{k}{\ell}\left(\frac{\del}{\del z}\right)^\ell\left(z^k H'(z)^{k-\ell}\right)\right|_{z=1},
\end{align*}
which is $m_k$ by \Cref{thm:momreg}.

Similarly,
\begin{align*}
     \oint_1\frac{1}{w-1}\left(wH'(w)+\frac{w}{w-1}\right)^k\frac{dw}{2\pi\i} & = \oint_1\sum_{\ell = 0}^k\dbinom{k}{\ell} \dfrac{w^kH'(w)^{k-\ell}}{(w-1)^{\ell+1}}\frac{dw}{2\pi\i} \\
     & = \sum_{\ell = 0}^k\dbinom{k}{\ell}\oint_1 \dfrac{w^kH'(w)^{k-\ell}}{(w-1)^{\ell+1}}\frac{dw}{2\pi\i} \\
     & = \left.\sum_{\ell=0}^k\frac{1}{\ell!}\binom{k}{\ell}\left(\frac{\del}{\del z}\right)^\ell\left(z^k H'(z)^{k-\ell}\right)\right|_{z=1},
\end{align*}
which is $d_k$ by \Cref{thm:momdiff}.
\end{proof}

To simplify our  expressions, define $F(z)$ and $G(z)$ as 
\[F(z)=\sum_{k=0}^\infty z^k\int_\mathbb{R}x^k \mathbf{m}(dx)=
\sum_{k=0}^\infty m_kz^k\]
and
\[G(z)=\sum_{k=1}^\infty\frac{z^k}{k}\int_\mathbb{R}x^k \mathbf{d}(dx)=
\sum_{k=1}^\infty d_k\frac{z^k}{k},\]
respectively, for the proof of Theorem \ref{thm:DMK}. We have the following lemma that allows us to work with the contour integrals.
\begin{lemma}
\label{lem:inverse}
Let
\[y(w)=\frac{1}{wH'(w)+\frac{w}{w-1}}.\]
This function is locally invertible in a neighborhood of $1$ with inverse
\[w=e^{yF(y)}.\]
\end{lemma}
\begin{proof}
In a neighborhood of $1$, the function $y(w)$ behaves like $\frac{w-1}{w}$ since $H$ is holomorphic, so $y'(1)=1$, and therefore, the inverse exists by the inverse function theorem. Furthermore, a counterclockwise contour of $w$ around $1$ is equivalent to a counterclockwise contour of $y$ around $0$. Thus, there is some holomorphic function $A(y)$ such that
\[\log w=A(y)\]
for $y$ in some neighborhood of $0$. Then, $\frac{dw}{w}=A'(y)dy$, so
\[m_k=\frac{1}{k+1}\oint_0\frac{1}{y^{k+1}}\frac{A'(y)dy}{2\pi\i},\]
which implies
\[m_k=\frac{1}{(k+1)!}\left[\left(\frac{\partial}{\partial y}\right)^kA'(y)\right]_{y=0}.\]
Since $\displaystyle\sum_{k=0}^\infty m_kz^k=F(z)$, we have
\[A(z)=zF(z),\]
as desired.
\end{proof}
\begin{rmk}
This lemma is essentially the same as combining equations (4.5) and (2.5) form \cite{BuG1}. We opt to prove the lemma directly here since our setup differs slightly from \cite{BuG1}, and our proof avoids mentioning $R$-transforms.
\end{rmk}

Now, we are ready to prove Theorem \ref{thm:DMK}.

\begin{proof}[Proof of Theorem \ref{thm:DMK}]
Using the notation and statement of Lemma \ref{lem:inverse}, we have
\[d_k=\oint_0\frac{e^{yF(y)}}{e^{yF(y)}-1}y^{-k}[F(y)+yF'(y)]\frac{dy}{2\pi\i}=\oint_0y^{-(k+1)}\frac{ye^{yF(y)}}{e^{yF(y)}-1}[F(y)+yF'(y)]\frac{dy}{2\pi\i}.\]
Applying Cauchy's Differentiation Formula yields
\[ d_k=\left.\frac{1}{k!}\left(\frac{\del}{\del y}\right)^k\left[\frac{ye^{yF(y)}}{e^{yF(y)}-1}[F(y)+yF'(y)]\right]\right|_{y=0}. \]
Thus,
\[1+\sum_{k=1}^{\infty}d_kz^k=\frac{ze^{zF(z)}}{e^{zF(z)}-1}[F(z)+zF'(z)],\]
so
\[1+zG'(z)=\frac{ze^{zF(z)}}{e^{zF(z)}-1}[F(z)+zF'(z)].\]
It is easy to see that this implies $e^{G(z)}=\frac{1}{z}\left(e^{zF(z)}-1\right)$, completing the proof of Theorem \ref{thm:DMK}.
\end{proof}

\section{Quantized Markov-Krein Correspondence} \label{sec:bijection_proof}
In this section, we prove \Cref{thm:BQMK} and heuristically show that \Cref{thm:BMK} can be recovered from a semiclassical limit of \Cref{thm:BQMK}.

\subsection{Proof of Theorem \ref{thm:BQMK}}
We will need the following theorem for our proof.
\begin{theorem}[\cite{KREIN}*{Theorem A.6}]
\label{thm:A6}
Let $\cR[a,b]$ be the class of complex functions $F(z)$ that satisfy the following conditions.
\begin{itemize}
    \item $F(z)$ is holomorphic for $\Im z>0$.
    \item $\Im F(z)\ge 0$ for $\Im z>0$.
    \item $F(z)$ is holomorphic and positive in the interval $(-\infty,a)$ and holomorphic and negative in the interval $(b,\infty)$.
\end{itemize}
Then, a function $F(z)$ is in the class $\cR[a,b]$ if and only if there is a nonnegative measure $\psi$ on $[a,b]$ such that
\[F(z)=\int_a^b\frac{d\psi(t)}{t-z}.\]
Furthermore, this representation is unique.
\end{theorem}
The proof of Theorem \ref{thm:BMK} can be found in \cite{KEROV}. We present the proof of Theorem \ref{thm:BQMK}, which is similar. The following simpler theorem can be used to prove Theorem \ref{thm:BQMK}.
\begin{theorem}
\label{thm:BQMK1}
There is a bijective correspondence between $\ctM[a,b]$ and $\{\mu\in\cM[a,b]:R_\mu(u)>-1\text{ for }u<a\}$ where $\psi\leftrightarrow\mu$ if and only if
\[R_\psi(u)=R_\mu(u).\]
\end{theorem}
Note that Theorems \ref{thm:BMK} and \ref{thm:BQMK1} immediately imply Theorem \ref{thm:BQMK}.

\subsubsection{Proof of $\psi\to\mu$}
Suppose we're given $\psi\in\ctM[a,b]$. Note that for $\Im u\ge 0$,
\[0\le\Im\int_a^b\frac{d\psi(t)}{t-u}<\int_{-\infty}^{\infty}\frac{\eta dt}{(t-\xi)^2+\eta^2}=\pi,\]
as $0\le\psi'(t)\le 1$. To show that we get a unique $\mu$ satisfying $R_\mu(u)=R_\psi(u)$, it suffices to show that
\[F(u):=1-\exp\left(-G(u)\right)\]
satisfies the conditions of Theorem \ref{thm:A6}, where
\[G(u):=\int_a^b\frac{d\psi(t)}{t-u}\]
is a function satisfying the conditions of Theorem \ref{thm:A6}, along with the restriction
\[0\le \Im G(u)<\pi\]
for $\Im u\ge 0$. Indeed, this is not hard to verify. Now,
\[R_\mu(u)=\exp\left(-G(u)\right)-1>-1,\]
so the proof of this direction is complete.

\subsubsection{Proof of $\mu\to\psi$}
Suppose we're given $\mu\in\{\mu\in\cM[a,b]:R_\mu(u)>-1\text{ for }u<a\}$. Let
\[-R_\mu(u)=F(u)=\int_a^b\frac{d\mu(t)}{t-u}.\]
Define
\[G(u)=-\log(1-F(u)),\]
which exists since $F(u)<1$ whenever $F(u)$ is real. The conditions of Theorem \ref{thm:A6} are satisfied, so there exists a unique measure $\psi\in\cM[a,b]$ such that
\[G(u)=\int_a^b\frac{d\psi(t)}{t-u}.\]
Since 
\[0\le \Im G(u)<\pi,\]
we have $\psi'(t)\le 1$ (by the Cauchy-Stieljis inversion formula), so we have $\psi\in\ctM[a,b]$, as desired. This completes the proof of Theorem \ref{thm:BQMK1}, and thus Theorem \ref{thm:BQMK}.

\subsection{Semiclassical Limit}
\label{ssec:semiclassical}
In fact, the Markov-Krein correspondence (\Cref{thm:BMK}) can be obtained with a semiclassical limit of the quantized Markov-Krein correspondence (\Cref{thm:BQMK}).

Specifically, for measures $\mu\in \cM[a,b]$ and $w\in\cD[a,b]$, it's necessary for the Markov-Krein correspondence to be a bijection. To proceed, let a sequence of probability measures $\psi_\ve$ satisfy
\[\psi_\ve\in \ctM[a/\ve, b/\ve]\]
for all $\ve > 0$. Define the sequence of probability measures $\hat{\mu}_\ve$ as $d\hat{\mu}_\ve(t) = d\psi_\ve(t/\ve)$, so that $\hat{\mu}_\ve\in\cM[a, b]$ and has density bounded by $1/\ve$. Define $w_\ve\in\ctD[a/\ve, b/\ve]$  to be the corresponding measures for $\psi_\ve$ and $\hat{w}_\ve\in\cD[a,b]$ as the measure with $d\hat{w}_\ve(t) = d w_\ve(t/\ve)$.

Under the limit $\ve\rightarrow 0,$ construct the $\psi_\ve$ so that the $\hat{\mu}_\ve$ converge to $\mu$; observe that the $\hat{w}_\ve$ converge to some $w$. Alternatively, we could have started by constructing the $w_\ve$ so that the $\hat{w}_\ve$ converge to some $w$, and determine the $\psi_\ve$ from the $w_\ve$. Now, each $\psi_\ve$ satisfies the quantized Markov-Krein correspondence, or 
\begin{equation}
\label{eq:QMK_eps1}
    \frac{1}{z}\left(-1+\exp\left(z\sum_{k=0}^\infty \psi_{\ve, k} z^k\right)\right)=\exp\left(\sum_{k=1}^\infty\frac{p_k(w_\ve)}{k}z^k\right).
\end{equation}
From setting the variable $z$ as $\ve z$ in (\ref{eq:QMK_eps1}), we know that
\begin{equation}
\label{eq:QMK_eps2}
\frac{1}{\ve z}\left(-1+\exp\left(\ve z\sum_{k=0}^\infty \psi_{\ve, k} (\ve z)^k\right)\right)=\exp\left(\sum_{k=1}^\infty\frac{p_k(w_{\ve})}{k}(\ve z)^k\right).
\end{equation}
Next, to simplify the expression define $F_{\ve}(z)$ as
\[F_{\ve}(z)= \sum_{k=0}^\infty \psi_{\ve, k} (\ve z)^k.\]
Also, note that $\ve^k\psi_{\ve, k}=\hat{\mu}_{\ve, k}$ and $\ve^k p_k(w_{\ve})=p_k(\hat{w}_{\ve}).$
The first implies that
\[ F_\ve(z)=\sum_{k=0}^\infty\psi_{\ve, k} \left(\ve z\right)^k = \sum_{k=0}^\infty\hat{\mu}_{\ve, k} z^k \to \sum_{k=0}^\infty\mu_k z^k,\]
and the second implies that
\[\sum_{k=1}^\infty\frac{p_k(w_{\ve})}{k}(\ve z)^k= \sum_{k=1}^\infty\frac{p_k(\hat{w}_{\ve})}{k} z^k \to \sum_{k=1}^\infty\frac{p_k(w)}{k} z^k,\]
under the limit $\ve\rightarrow 0$. 
Now, we use the Taylor Series expansion to obtain 
\[\frac{1}{\ve z}\left(-1+\exp(\ve zF_{\ve}(z))\right) = F_{\ve}(z)+\sum_{k=2}^\infty \frac{(\ve z)^{k-1}F_{\ve}(z)^k}{k!}\rightarrow F_{\ve}(z)\]
with the limit $\ve\rightarrow 0$. Here, as $F_\ve(z)$ becomes a holomorphic function, the terms with powers of $\ve$ at least one are removed. Therefore, (\ref{eq:QMK_eps2}) becomes, under the limit,
\[F_\ve(z)=\exp\left(\sum_{k=1}^\infty\frac{p_k(w_{\ve})}{k}(\ve z)^k\right)=\exp\left(\sum_{k=1}^\infty\frac{p_k(\hat{w}_{\ve})}{k} z^k\right),\]
or
\[
\sum_{k=0}^\infty\mu_k z^k= \exp\left( \sum_{k=1}^\infty\frac{p_k(w)}{k}z^k\right),
\]
which is the Markov-Krein correspondence. 

Assume that there exist two measures $w_1, w_2\in\cD[a,b]$ satisfying the Markov-Krein correspondence. Then, the $p_k(w_1)=p_k(w_2)$ for all positive integers $k$. But, since $w_1''$ and $w_2''$ are supported on $[a,b]$, \[\sum_{k=1}^\infty\frac{p_k(w)}{k}z^k\]
would have a positive radius of convergence for $w=w_1$ and $w=w_2$. Therefore, $w_1''=w_2''$ and $w_1=w_2$. From this, there is an unique measure $w$ satisfying the Markov-Krein correspondence for each $\mu$. By constructing the $w_\ve$ first, the other direction can also be shown to hold.

\appendix\section{Example of Quantized Markov-Krein Correspondence}
\label{sec:appendixb}

\begin{figure}[!htbp]
\[\includegraphics[scale=2]{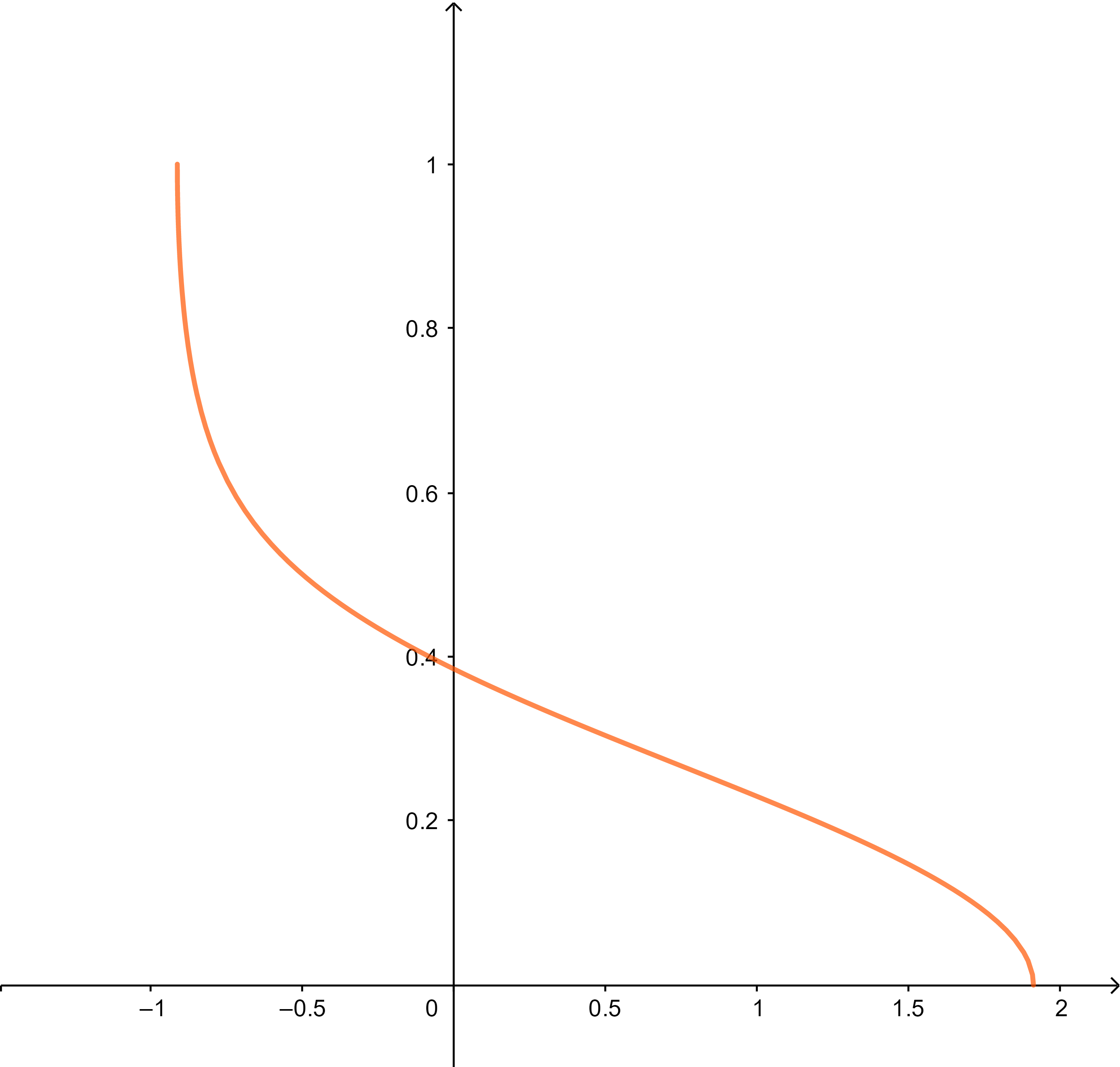}\quad\includegraphics[scale=2]{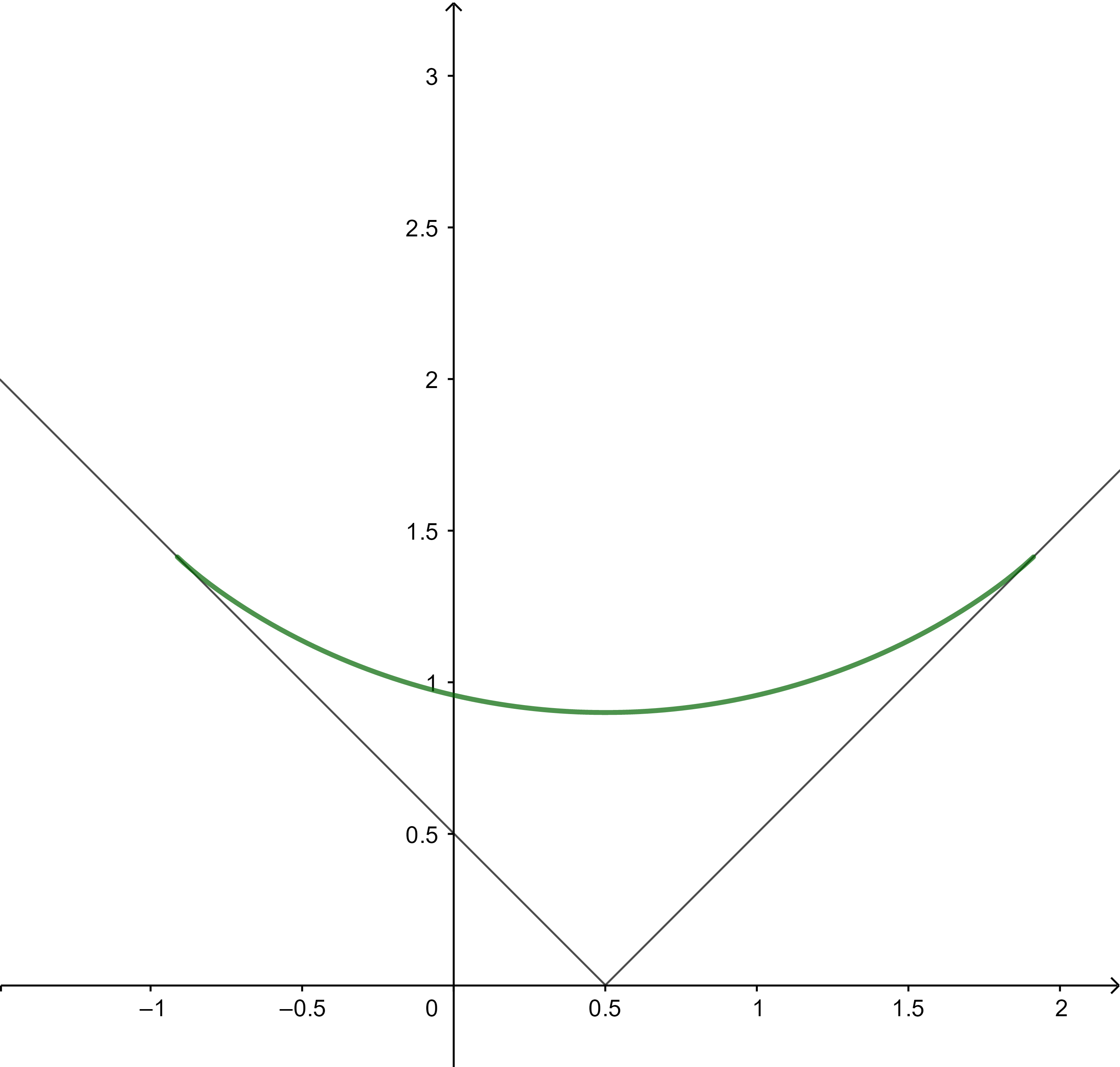}\]
\[\includegraphics[scale=1]{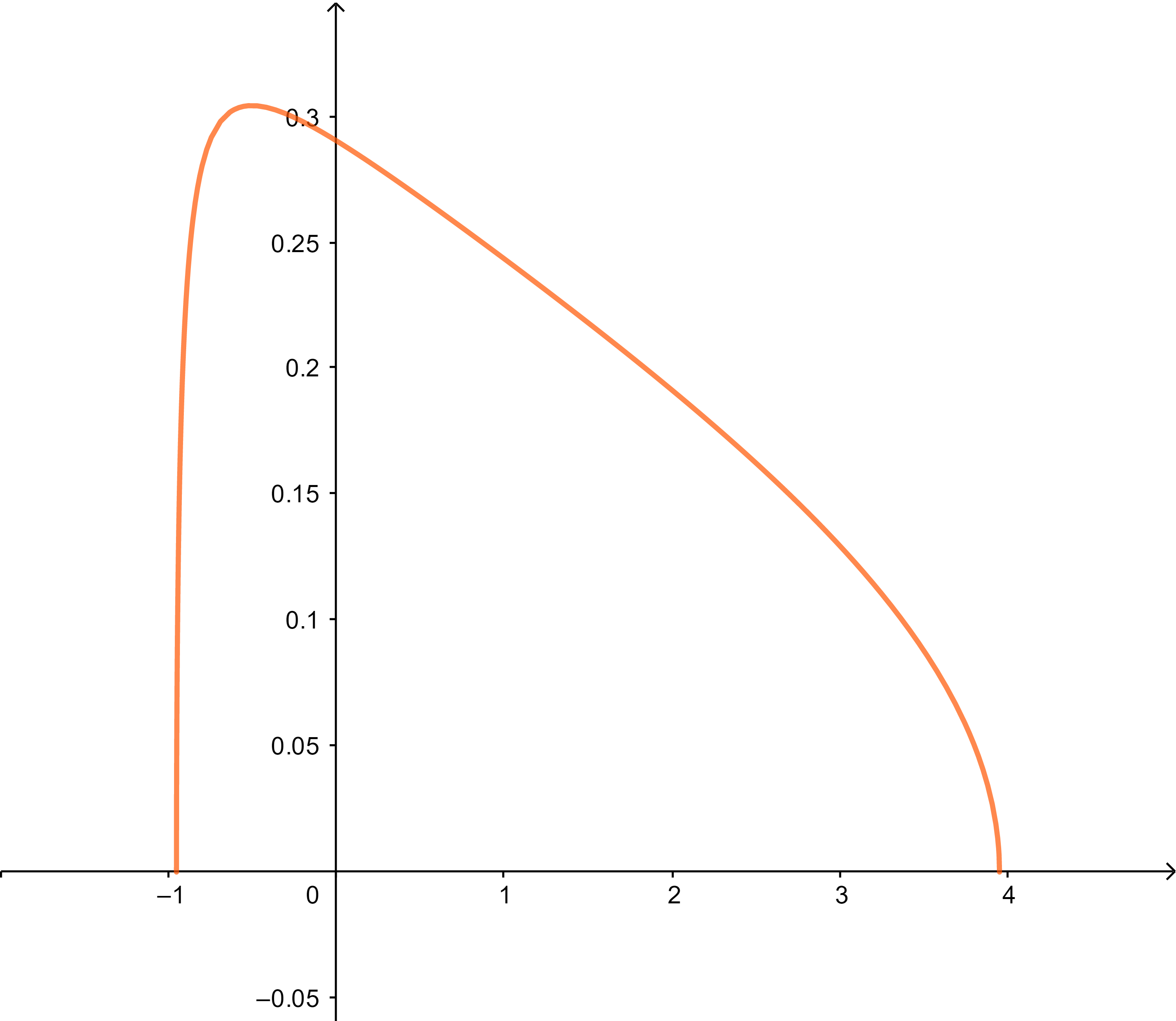}\quad\includegraphics[scale=1]{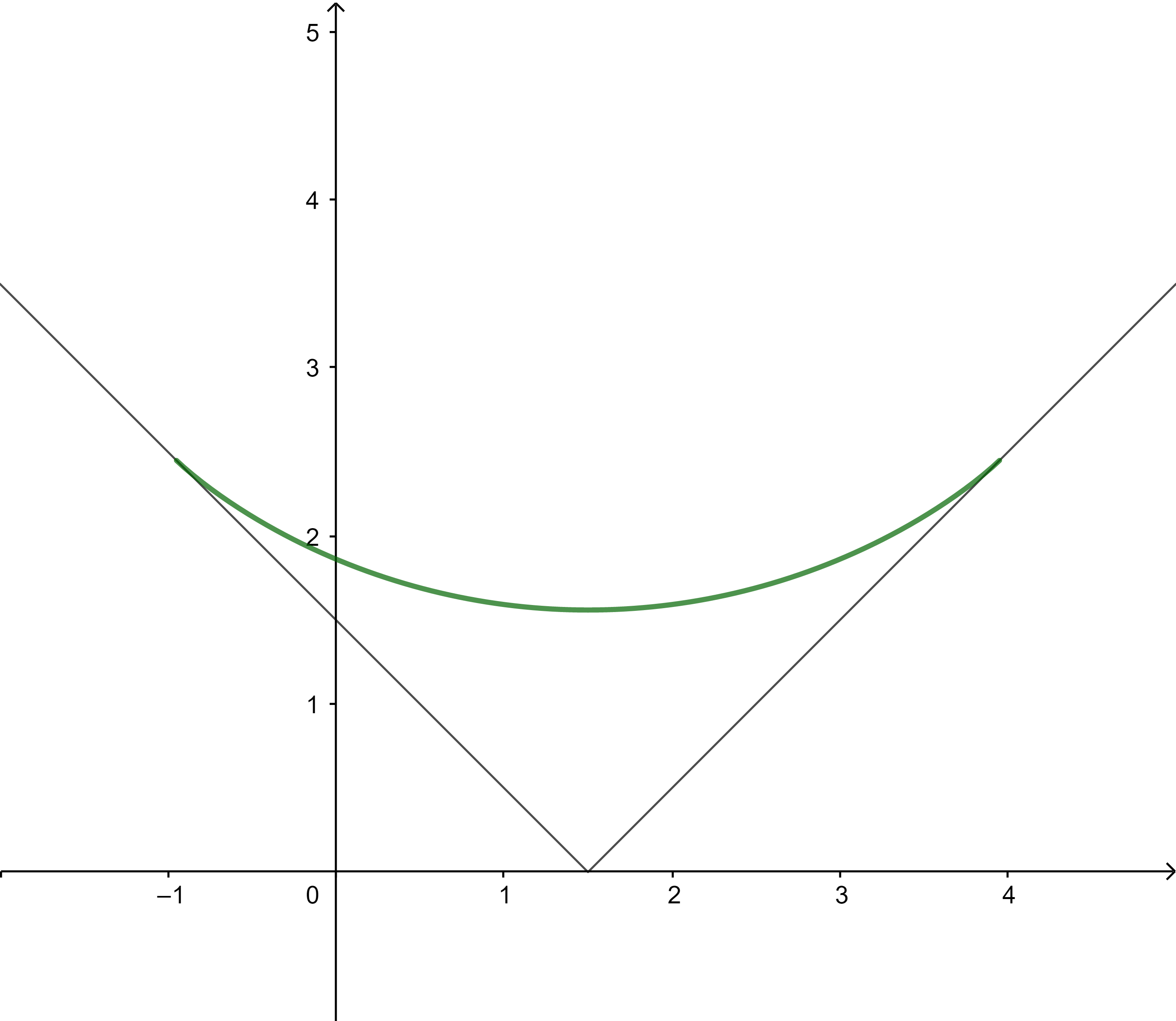}\]
\[\includegraphics[scale=0.8]{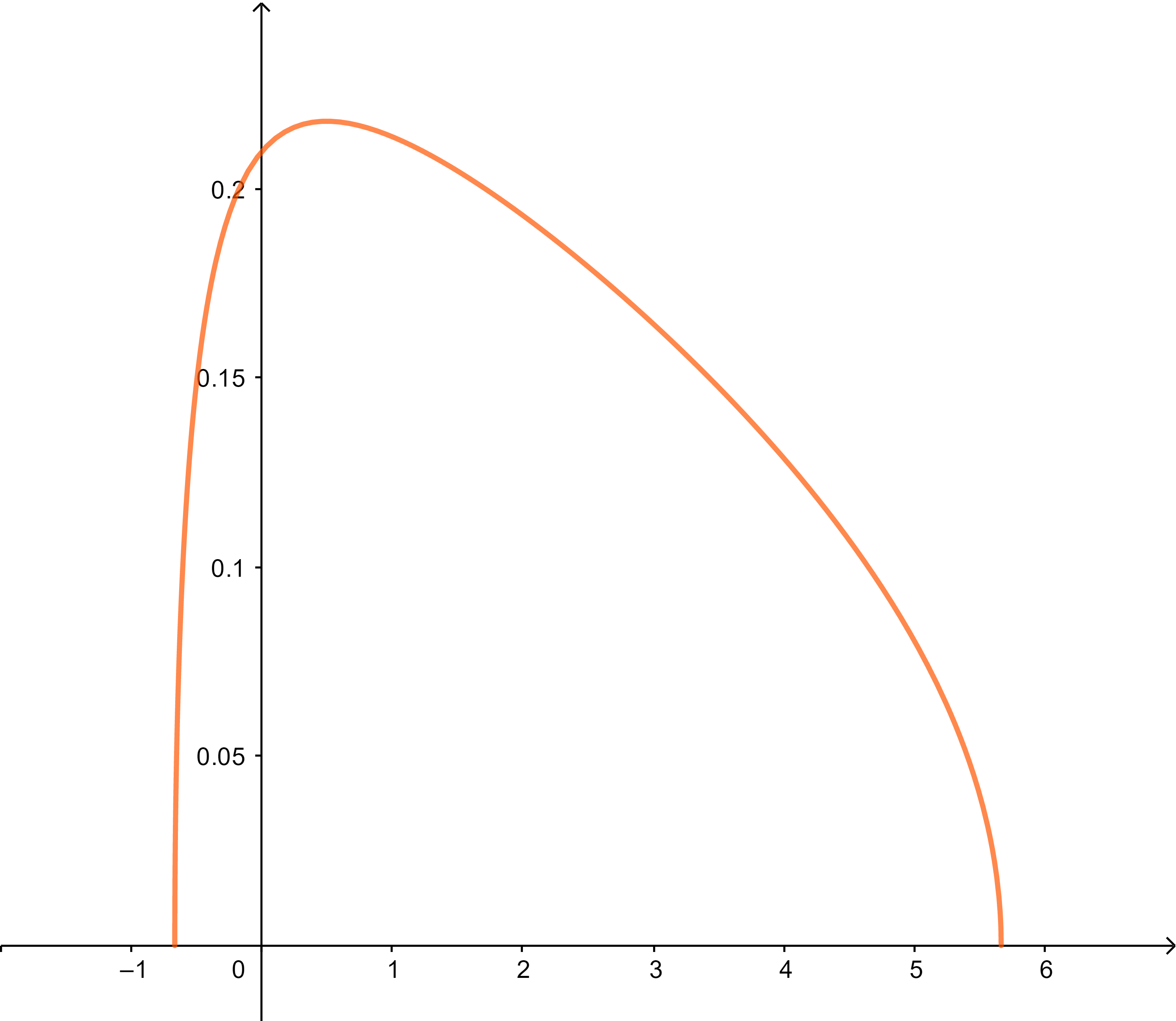}\quad\includegraphics[scale=0.8]{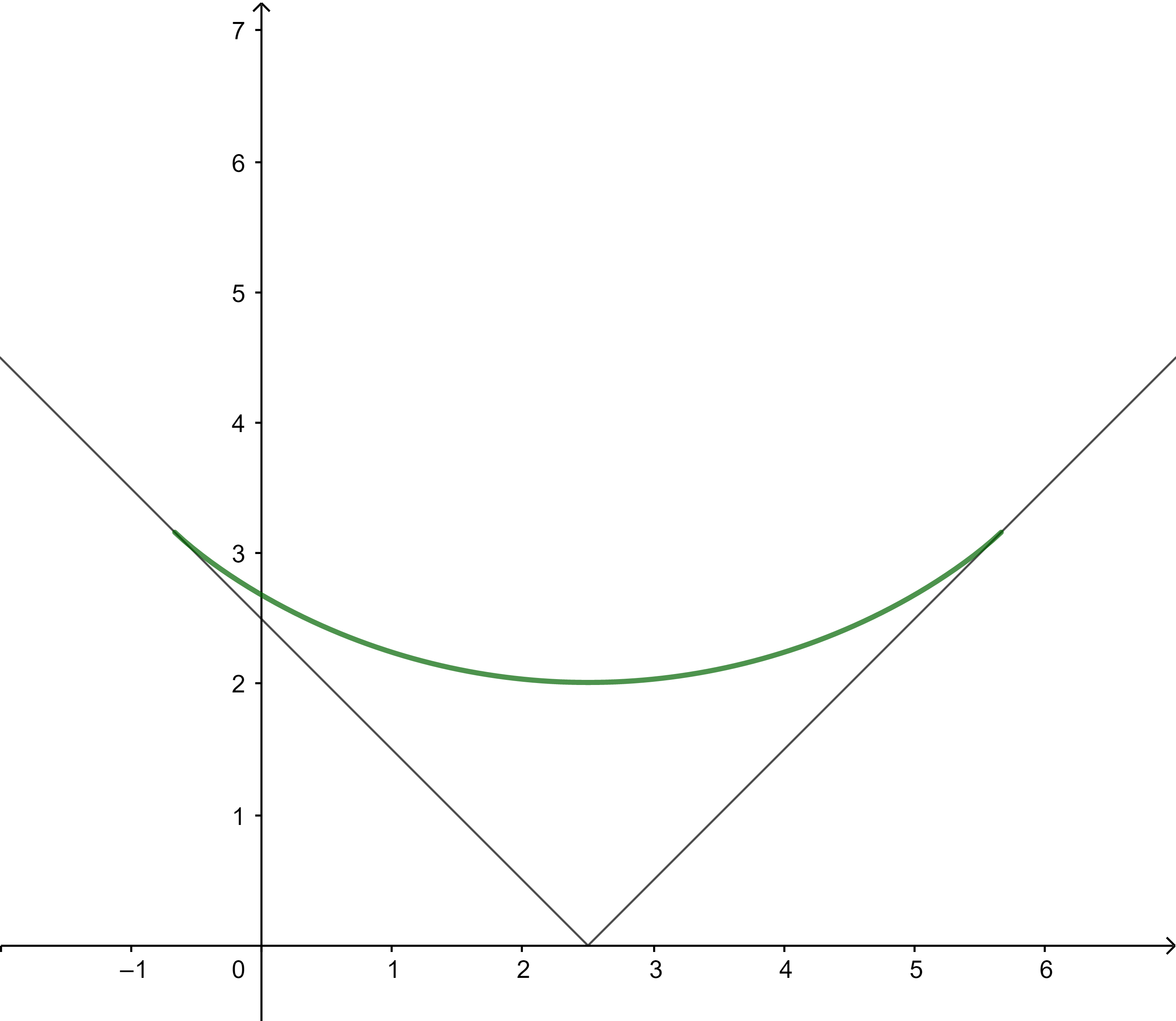}\]
\caption{One-sided Plancherel character limiting measure and corresponding diagram for $\gamma=0.5,1.5,2.5$}
\label{fig:plancherel}
\end{figure}

The semicircle law and the VKLS curve are naturally paired under the Markov-Krein correspondence. In the quantized Makrov-Krein correspondence, as the quantized analgoue of the semicircle law, we can let the measure $\psi\in\ctM[a,b]$ be the \textit{one-sided Plancherel character} with parameter $\gamma>0$ (see \cite{BBO}*{Appendix A}), and find the corresponding diagram $w\in\ctD[a,b]$ under \Cref{thm:BQMK}.

The density of $\psi$ is given by
\[\psi(x) = \frac{1}{\pi}\arccos\frac{x+\gamma}{2\sqrt{\gamma(x+1)}}\quad\text{for }x\in[\gamma-2\sqrt{\gamma},\gamma+2\sqrt{\gamma}].\]
After computations, we obtain that the corresponding diagram is given by
\[w(x) = \frac{2}{\pi}\left((x-\gamma)\arcsin\left(\frac{x-\gamma}{2\sqrt{\gamma}}\right)+\sqrt{4\gamma-(x-\gamma)^2}\right).\]
Note that this is a shifted and scaled version of the VKLS curve. There are examples of the one-sided Plancherel character and $w(x)$ in Figure \ref{fig:plancherel}.

\newpage
\bibliography{mybib.bib}

\end{document}

%% file: SemicircleVKLSCurve.tex
\begin{center}
\definecolor{uuuuuu}{rgb}{0.26666666666666666,0.26666666666666666,0.26666666666666666}
\begin{tikzpicture}[line cap=round,line join=round,>=triangle 45,x=1.5cm,y=1.5cm]
\clip(-7.5,-0.5) rectangle (2.5,2.5);
\draw[line width=1.5pt] (-1.9999900000000106,1.9999900134211368) -- (-1.9999900000000106,1.9999900134211368);
\draw[line width=1.5pt] (-1.9999900000000106,1.9999900134211368) -- (-1.9899900322222106,1.9904151865838917);
\draw[line width=1.5pt] (-1.9899900322222106,1.9904151865838917) -- (-1.9799900644444106,1.9811919827497833);
\draw[line width=1.5pt] (-1.9799900644444106,1.9811919827497833) -- (-1.9699900966666106,1.9721981636950292);
\draw[line width=1.5pt] (-1.9699900966666106,1.9721981636950292) -- (-1.9599901288888106,1.963390099622217);
\draw[line width=1.5pt] (-1.9599901288888106,1.963390099622217) -- (-1.9499901611110106,1.9547426036528328);
\draw[line width=1.5pt] (-1.9499901611110106,1.9547426036528328) -- (-1.9399901933332107,1.946238702689093);
\draw[line width=1.5pt] (-1.9399901933332107,1.946238702689093) -- (-1.9299902255554107,1.9378659523520791);
\draw[line width=1.5pt] (-1.9299902255554107,1.9378659523520791) -- (-1.9199902577776107,1.9296147233637524);
\draw[line width=1.5pt] (-1.9199902577776107,1.9296147233637524) -- (-1.9099902899998107,1.9214772789869918);
\draw[line width=1.5pt] (-1.9099902899998107,1.9214772789869918) -- (-1.8999903222220107,1.9134472276015893);
\draw[line width=1.5pt] (-1.8999903222220107,1.9134472276015893) -- (-1.8899903544442107,1.9055191743357838);
\draw[line width=1.5pt] (-1.8899903544442107,1.9055191743357838) -- (-1.8799903866664107,1.8976884873085622);
\draw[line width=1.5pt] (-1.8799903866664107,1.8976884873085622) -- (-1.8699904188886107,1.8899511341089477);
\draw[line width=1.5pt] (-1.8699904188886107,1.8899511341089477) -- (-1.8599904511108107,1.8823035634964809);
\draw[line width=1.5pt] (-1.8599904511108107,1.8823035634964809) -- (-1.8499904833330107,1.8747426174083595);
\draw[line width=1.5pt] (-1.8499904833330107,1.8747426174083595) -- (-1.8399905155552108,1.8672654639650819);
\draw[line width=1.5pt] (-1.8399905155552108,1.8672654639650819) -- (-1.8299905477774108,1.8598695454397194);
\draw[line width=1.5pt] (-1.8299905477774108,1.8598695454397194) -- (-1.8199905799996108,1.8525525371497304);
\draw[line width=1.5pt] (-1.8199905799996108,1.8525525371497304) -- (-1.8099906122218108,1.845312314489222);
\draw[line width=1.5pt] (-1.8099906122218108,1.845312314489222) -- (-1.7999906444440108,1.8381469261396444);
\draw[line width=1.5pt] (-1.7999906444440108,1.8381469261396444) -- (-1.7899906766662108,1.831054572045682);
\draw[line width=1.5pt] (-1.7899906766662108,1.831054572045682) -- (-1.7799907088884108,1.8240335851192133);
\draw[line width=1.5pt] (-1.7799907088884108,1.8240335851192133) -- (-1.7699907411106108,1.8170824158974297);
\draw[line width=1.5pt] (-1.7699907411106108,1.8170824158974297) -- (-1.7599907733328108,1.8101996195689456);
\draw[line width=1.5pt] (-1.7599907733328108,1.8101996195689456) -- (-1.7499908055550109,1.8033838449179027);
\draw[line width=1.5pt] (-1.7499908055550109,1.8033838449179027) -- (-1.7399908377772109,1.796633824836362);
\draw[line width=1.5pt] (-1.7399908377772109,1.796633824836362) -- (-1.7299908699994109,1.7899483681301813);
\draw[line width=1.5pt] (-1.7299908699994109,1.7899483681301813) -- (-1.7199909022216109,1.7833263524002272);
\draw[line width=1.5pt] (-1.7199909022216109,1.7833263524002272) -- (-1.709990934443811,1.7767667178241062);
\draw[line width=1.5pt] (-1.709990934443811,1.7767667178241062) -- (-1.699990966666011,1.7702684616971374);
\draw[line width=1.5pt] (-1.699990966666011,1.7702684616971374) -- (-1.689990998888211,1.763830633617446);
\draw[line width=1.5pt] (-1.689990998888211,1.763830633617446) -- (-1.679991031110411,1.7574523312207104);
\draw[line width=1.5pt] (-1.679991031110411,1.7574523312207104) -- (-1.669991063332611,1.751132696386482);
\draw[line width=1.5pt] (-1.669991063332611,1.751132696386482) -- (-1.659991095554811,1.7448709118511458);
\draw[line width=1.5pt] (-1.659991095554811,1.7448709118511458) -- (-1.649991127777011,1.738666198173185);
\draw[line width=1.5pt] (-1.649991127777011,1.738666198173185) -- (-1.639991159999211,1.7325178110050259);
\draw[line width=1.5pt] (-1.639991159999211,1.7325178110050259) -- (-1.629991192221411,1.726425038632805);
\draw[line width=1.5pt] (-1.629991192221411,1.726425038632805) -- (-1.619991224443611,1.7203871997511702);
\draw[line width=1.5pt] (-1.619991224443611,1.7203871997511702) -- (-1.609991256665811,1.714403641445052);
\draw[line width=1.5pt] (-1.609991256665811,1.714403641445052) -- (-1.599991288888011,1.7084737373543148);
\draw[line width=1.5pt] (-1.599991288888011,1.7084737373543148) -- (-1.589991321110211,1.7025968860005454);
\draw[line width=1.5pt] (-1.589991321110211,1.7025968860005454) -- (-1.579991353332411,1.6967725092580366);
\draw[line width=1.5pt] (-1.579991353332411,1.6967725092580366) -- (-1.569991385554611,1.6910000509533984);
\draw[line width=1.5pt] (-1.569991385554611,1.6910000509533984) -- (-1.559991417776811,1.6852789755802164);
\draw[line width=1.5pt] (-1.559991417776811,1.6852789755802164) -- (-1.549991449999011,1.67960876711692);
\draw[line width=1.5pt] (-1.549991449999011,1.67960876711692) -- (-1.539991482221211,1.6739889279374427);
\draw[line width=1.5pt] (-1.539991482221211,1.6739889279374427) -- (-1.529991514443411,1.6684189778055396);
\draw[line width=1.5pt] (-1.529991514443411,1.6684189778055396) -- (-1.519991546665611,1.6628984529446926);
\draw[line width=1.5pt] (-1.519991546665611,1.6628984529446926) -- (-1.509991578887811,1.6574269051764607);
\draw[line width=1.5pt] (-1.509991578887811,1.6574269051764607) -- (-1.4999916111100111,1.6520039011209522);
\draw[line width=1.5pt] (-1.4999916111100111,1.6520039011209522) -- (-1.4899916433322111,1.6466290214537789);
\draw[line width=1.5pt] (-1.4899916433322111,1.6466290214537789) -- (-1.4799916755544111,1.6413018602144853);
\draw[line width=1.5pt] (-1.4799916755544111,1.6413018602144853) -- (-1.4699917077766111,1.63602202416196);
\draw[line width=1.5pt] (-1.4699917077766111,1.63602202416196) -- (-1.4599917399988112,1.6307891321728132);
\draw[line width=1.5pt] (-1.4599917399988112,1.6307891321728132) -- (-1.4499917722210112,1.6256028146791208);
\draw[line width=1.5pt] (-1.4499917722210112,1.6256028146791208) -- (-1.4399918044432112,1.6204627131422928);
\draw[line width=1.5pt] (-1.4399918044432112,1.6204627131422928) -- (-1.4299918366654112,1.6153684795601362);
\draw[line width=1.5pt] (-1.4299918366654112,1.6153684795601362) -- (-1.4199918688876112,1.610319776004485);
\draw[line width=1.5pt] (-1.4199918688876112,1.610319776004485) -- (-1.4099919011098112,1.6053162741870066);
\draw[line width=1.5pt] (-1.4099919011098112,1.6053162741870066) -- (-1.3999919333320112,1.600357655051022);
\draw[line width=1.5pt] (-1.3999919333320112,1.600357655051022) -- (-1.3899919655542112,1.5954436083873753);
\draw[line width=1.5pt] (-1.3899919655542112,1.5954436083873753) -- (-1.3799919977764112,1.5905738324725713);
\draw[line width=1.5pt] (-1.3799919977764112,1.5905738324725713) -- (-1.3699920299986112,1.5857480337275551);
\draw[line width=1.5pt] (-1.3699920299986112,1.5857480337275551) -- (-1.3599920622208113,1.5809659263956417);
\draw[line width=1.5pt] (-1.3599920622208113,1.5809659263956417) -- (-1.3499920944430113,1.5762272322382593);
\draw[line width=1.5pt] (-1.3499920944430113,1.5762272322382593) -- (-1.3399921266652113,1.5715316802472457);
\draw[line width=1.5pt] (-1.3399921266652113,1.5715316802472457) -- (-1.3299921588874113,1.5668790063725766);
\draw[line width=1.5pt] (-1.3299921588874113,1.5668790063725766) -- (-1.3199921911096113,1.5622689532644796);
\draw[line width=1.5pt] (-1.3199921911096113,1.5622689532644796) -- (-1.3099922233318113,1.5577012700289747);
\draw[line width=1.5pt] (-1.3099922233318113,1.5577012700289747) -- (-1.2999922555540113,1.5531757119959628);
\draw[line width=1.5pt] (-1.2999922555540113,1.5531757119959628) -- (-1.2899922877762113,1.5486920404990538);
\draw[line width=1.5pt] (-1.2899922877762113,1.5486920404990538) -- (-1.2799923199984113,1.544250022666378);
\draw[line width=1.5pt] (-1.2799923199984113,1.544250022666378) -- (-1.2699923522206114,1.5398494312217048);
\draw[line width=1.5pt] (-1.2699923522206114,1.5398494312217048) -- (-1.2599923844428114,1.5354900442952117);
\draw[line width=1.5pt] (-1.2599923844428114,1.5354900442952117) -- (-1.2499924166650114,1.531171645243336);
\draw[line width=1.5pt] (-1.2499924166650114,1.531171645243336) -- (-1.2399924488872114,1.5268940224771383);
\draw[line width=1.5pt] (-1.2399924488872114,1.5268940224771383) -- (-1.2299924811094114,1.5226569692986927);
\draw[line width=1.5pt] (-1.2299924811094114,1.5226569692986927) -- (-1.2199925133316114,1.518460283745014);
\draw[line width=1.5pt] (-1.2199925133316114,1.518460283745014) -- (-1.2099925455538114,1.5143037684391045);
\draw[line width=1.5pt] (-1.2099925455538114,1.5143037684391045) -- (-1.1999925777760114,1.5101872304476907);
\draw[line width=1.5pt] (-1.1999925777760114,1.5101872304476907) -- (-1.1899926099982114,1.5061104811452946);
\draw[line width=1.5pt] (-1.1899926099982114,1.5061104811452946) -- (-1.1799926422204114,1.5020733360842693);
\draw[line width=1.5pt] (-1.1799926422204114,1.5020733360842693) -- (-1.1699926744426115,1.49807561487048);
\draw[line width=1.5pt] (-1.1699926744426115,1.49807561487048) -- (-1.1599927066648115,1.4941171410443186);
\draw[line width=1.5pt] (-1.1599927066648115,1.4941171410443186) -- (-1.1499927388870115,1.4901977419667674);
\draw[line width=1.5pt] (-1.1499927388870115,1.4901977419667674) -- (-1.1399927711092115,1.4863172487102434);
\draw[line width=1.5pt] (-1.1399927711092115,1.4863172487102434) -- (-1.1299928033314115,1.4824754959539708);
\draw[line width=1.5pt] (-1.1299928033314115,1.4824754959539708) -- (-1.1199928355536115,1.4786723218836504);
\draw[line width=1.5pt] (-1.1199928355536115,1.4786723218836504) -- (-1.1099928677758115,1.4749075680952);
\draw[line width=1.5pt] (-1.1099928677758115,1.4749075680952) -- (-1.0999928999980115,1.471181079502366);
\draw[line width=1.5pt] (-1.0999928999980115,1.471181079502366) -- (-1.0899929322202115,1.467492704248009);
\draw[line width=1.5pt] (-1.0899929322202115,1.467492704248009) -- (-1.0799929644424116,1.463842293618878);
\draw[line width=1.5pt] (-1.0799929644424116,1.463842293618878) -- (-1.0699929966646116,1.4602297019637087);
\draw[line width=1.5pt] (-1.0699929966646116,1.4602297019637087) -- (-1.0599930288868116,1.4566547866144843);
\draw[line width=1.5pt] (-1.0599930288868116,1.4566547866144843) -- (-1.0499930611090116,1.4531174078106992);
\draw[line width=1.5pt] (-1.0499930611090116,1.4531174078106992) -- (-1.0399930933312116,1.4496174286264942);
\draw[line width=1.5pt] (-1.0399930933312116,1.4496174286264942) -- (-1.0299931255534116,1.4461547149005216);
\draw[line width=1.5pt] (-1.0299931255534116,1.4461547149005216) -- (-1.0199931577756116,1.4427291351684157);
\draw[line width=1.5pt] (-1.0199931577756116,1.4427291351684157) -- (-1.0099931899978116,1.4393405605977492);
\draw[line width=1.5pt] (-1.0099931899978116,1.4393405605977492) -- (-0.9999932222200116,1.4359888649253638);
\draw[line width=1.5pt] (-0.9999932222200116,1.4359888649253638) -- (-0.9899932544422116,1.4326739243969633);
\draw[line width=1.5pt] (-0.9899932544422116,1.4326739243969633) -- (-0.9799932866644117,1.4293956177088782);
\draw[line width=1.5pt] (-0.9799932866644117,1.4293956177088782) -- (-0.9699933188866117,1.426153825951895);
\draw[line width=1.5pt] (-0.9699933188866117,1.426153825951895) -- (-0.9599933511088117,1.4229484325570705);
\draw[line width=1.5pt] (-0.9599933511088117,1.4229484325570705) -- (-0.9499933833310117,1.4197793232434357);
\draw[line width=1.5pt] (-0.9499933833310117,1.4197793232434357) -- (-0.9399934155532117,1.4166463859675198);
\draw[line width=1.5pt] (-0.9399934155532117,1.4166463859675198) -- (-0.9299934477754117,1.413549510874607);
\draw[line width=1.5pt] (-0.9299934477754117,1.413549510874607) -- (-0.9199934799976117,1.4104885902516593);
\draw[line width=1.5pt] (-0.9199934799976117,1.4104885902516593) -- (-0.9099935122198117,1.407463518481838);
\draw[line width=1.5pt] (-0.9099935122198117,1.407463518481838) -- (-0.8999935444420117,1.4044741920005508);
\draw[line width=1.5pt] (-0.8999935444420117,1.4044741920005508) -- (-0.8899935766642117,1.4015205092529712);
\draw[line width=1.5pt] (-0.8899935766642117,1.4015205092529712) -- (-0.8799936088864118,1.3986023706529676);
\draw[line width=1.5pt] (-0.8799936088864118,1.3986023706529676) -- (-0.8699936411086118,1.395719678543381);
\draw[line width=1.5pt] (-0.8699936411086118,1.395719678543381) -- (-0.8599936733308118,1.392872337157611);
\draw[line width=1.5pt] (-0.8599936733308118,1.392872337157611) -- (-0.8499937055530118,1.3900602525824428);
\draw[line width=1.5pt] (-0.8499937055530118,1.3900602525824428) -- (-0.8399937377752118,1.3872833327220815);
\draw[line width=1.5pt] (-0.8399937377752118,1.3872833327220815) -- (-0.8299937699974118,1.3845414872633384);
\draw[line width=1.5pt] (-0.8299937699974118,1.3845414872633384) -- (-0.8199938022196118,1.3818346276419304);
\draw[line width=1.5pt] (-0.8199938022196118,1.3818346276419304) -- (-0.8099938344418118,1.3791626670098505);
\draw[line width=1.5pt] (-0.8099938344418118,1.3791626670098505) -- (-0.7999938666640118,1.3765255202037634);
\draw[line width=1.5pt] (-0.7999938666640118,1.3765255202037634) -- (-0.7899938988862119,1.3739231037144013);
\draw[line width=1.5pt] (-0.7899938988862119,1.3739231037144013) -- (-0.7799939311084119,1.3713553356569093);
\draw[line width=1.5pt] (-0.7799939311084119,1.3713553356569093) -- (-0.7699939633306119,1.3688221357421138);
\draw[line width=1.5pt] (-0.7699939633306119,1.3688221357421138) -- (-0.7599939955528119,1.366323425248684);
\draw[line width=1.5pt] (-0.7599939955528119,1.366323425248684) -- (-0.7499940277750119,1.3638591269961469);
\draw[line width=1.5pt] (-0.7499940277750119,1.3638591269961469) -- (-0.7399940599972119,1.361429165318729);
\draw[line width=1.5pt] (-0.7399940599972119,1.361429165318729) -- (-0.7299940922194119,1.3590334660400005);
\draw[line width=1.5pt] (-0.7299940922194119,1.3590334660400005) -- (-0.7199941244416119,1.3566719564482885);
\draw[line width=1.5pt] (-0.7199941244416119,1.3566719564482885) -- (-0.7099941566638119,1.354344565272836);
\draw[line width=1.5pt] (-0.7099941566638119,1.354344565272836) -- (-0.699994188886012,1.3520512226606813);
\draw[line width=1.5pt] (-0.699994188886012,1.3520512226606813) -- (-0.689994221108212,1.3497918601542327);
\draw[line width=1.5pt] (-0.689994221108212,1.3497918601542327) -- (-0.679994253330412,1.3475664106695193);
\draw[line width=1.5pt] (-0.679994253330412,1.3475664106695193) -- (-0.669994285552612,1.3453748084750887);
\draw[line width=1.5pt] (-0.669994285552612,1.3453748084750887) -- (-0.659994317774812,1.3432169891715398);
\draw[line width=1.5pt] (-0.659994317774812,1.3432169891715398) -- (-0.649994349997012,1.3410928896716627);
\draw[line width=1.5pt] (-0.649994349997012,1.3410928896716627) -- (-0.639994382219212,1.3390024481811709);
\draw[line width=1.5pt] (-0.639994382219212,1.3390024481811709) -- (-0.629994414441412,1.3369456041800063);
\draw[line width=1.5pt] (-0.629994414441412,1.3369456041800063) -- (-0.619994446663612,1.334922298404198);
\draw[line width=1.5pt] (-0.619994446663612,1.334922298404198) -- (-0.609994478885812,1.3329324728282603);
\draw[line width=1.5pt] (-0.609994478885812,1.3329324728282603) -- (-0.599994511108012,1.3309760706481115);
\draw[line width=1.5pt] (-0.599994511108012,1.3309760706481115) -- (-0.5899945433302121,1.3290530362644988);
\draw[line width=1.5pt] (-0.5899945433302121,1.3290530362644988) -- (-0.5799945755524121,1.3271633152669156);
\draw[line width=1.5pt] (-0.5799945755524121,1.3271633152669156) -- (-0.5699946077746121,1.3253068544179936);
\draw[line width=1.5pt] (-0.5699946077746121,1.3253068544179936) -- (-0.5599946399968121,1.3234836016383595);
\draw[line width=1.5pt] (-0.5599946399968121,1.3234836016383595) -- (-0.5499946722190121,1.3216935059919412);
\draw[line width=1.5pt] (-0.5499946722190121,1.3216935059919412) -- (-0.5399947044412121,1.319936517671711);
\draw[line width=1.5pt] (-0.5399947044412121,1.319936517671711) -- (-0.5299947366634121,1.3182125879858537);
\draw[line width=1.5pt] (-0.5299947366634121,1.3182125879858537) -- (-0.5199947688856121,1.3165216693443487);
\draw[line width=1.5pt] (-0.5199947688856121,1.3165216693443487) -- (-0.5099948011078121,1.3148637152459537);
\draw[line width=1.5pt] (-0.5099948011078121,1.3148637152459537) -- (-0.49999483333001216,1.313238680265579);
\draw[line width=1.5pt] (-0.49999483333001216,1.313238680265579) -- (-0.48999486555221217,1.3116465200420473);
\draw[line width=1.5pt] (-0.48999486555221217,1.3116465200420473) -- (-0.4799948977744122,1.3100871912662162);
\draw[line width=1.5pt] (-0.4799948977744122,1.3100871912662162) -- (-0.4699949299966122,1.3085606516694714);
\draw[line width=1.5pt] (-0.4699949299966122,1.3085606516694714) -- (-0.4599949622188122,1.3070668600125674);
\draw[line width=1.5pt] (-0.4599949622188122,1.3070668600125674) -- (-0.4499949944410122,1.3056057760748128);
\draw[line width=1.5pt] (-0.4499949944410122,1.3056057760748128) -- (-0.4399950266632122,1.30417736064359);
\draw[line width=1.5pt] (-0.4399950266632122,1.30417736064359) -- (-0.42999505888541223,1.3027815755042051);
\draw[line width=1.5pt] (-0.42999505888541223,1.3027815755042051) -- (-0.41999509110761224,1.3014183834300546);
\draw[line width=1.5pt] (-0.41999509110761224,1.3014183834300546) -- (-0.40999512332981225,1.3000877481731044);
\draw[line width=1.5pt] (-0.40999512332981225,1.3000877481731044) -- (-0.39999515555201226,1.2987896344546743);
\draw[line width=1.5pt] (-0.39999515555201226,1.2987896344546743) -- (-0.38999518777421227,1.29752400795652);
\draw[line width=1.5pt] (-0.38999518777421227,1.29752400795652) -- (-0.3799952199964123,1.2962908353122071);
\draw[line width=1.5pt] (-0.3799952199964123,1.2962908353122071) -- (-0.3699952522186123,1.2950900840987682);
\draw[line width=1.5pt] (-0.3699952522186123,1.2950900840987682) -- (-0.3599952844408123,1.2939217228286422);
\draw[line width=1.5pt] (-0.3599952844408123,1.2939217228286422) -- (-0.3499953166630123,1.2927857209418827);
\draw[line width=1.5pt] (-0.3499953166630123,1.2927857209418827) -- (-0.3399953488852123,1.2916820487986365);
\draw[line width=1.5pt] (-0.3399953488852123,1.2916820487986365) -- (-0.32999538110741233,1.2906106776718842);
\draw[line width=1.5pt] (-0.32999538110741233,1.2906106776718842) -- (-0.31999541332961234,1.2895715797404372);
\draw[line width=1.5pt] (-0.31999541332961234,1.2895715797404372) -- (-0.30999544555181235,1.2885647280821844);
\draw[line width=1.5pt] (-0.30999544555181235,1.2885647280821844) -- (-0.29999547777401236,1.2875900966675924);
\draw[line width=1.5pt] (-0.29999547777401236,1.2875900966675924) -- (-0.2899955099962124,1.2866476603534405);
\draw[line width=1.5pt] (-0.2899955099962124,1.2866476603534405) -- (-0.2799955422184124,1.2857373948768012);
\draw[line width=1.5pt] (-0.2799955422184124,1.2857373948768012) -- (-0.2699955744406124,1.2848592768492513);
\draw[line width=1.5pt] (-0.2699955744406124,1.2848592768492513) -- (-0.2599956066628124,1.2840132837513136);
\draw[line width=1.5pt] (-0.2599956066628124,1.2840132837513136) -- (-0.2499956388850124,1.2831993939271271);
\draw[line width=1.5pt] (-0.2499956388850124,1.2831993939271271) -- (-0.23999567110721237,1.2824175865793384);
\draw[line width=1.5pt] (-0.23999567110721237,1.2824175865793384) -- (-0.22999570332941235,1.281667841764213);
\draw[line width=1.5pt] (-0.22999570332941235,1.281667841764213) -- (-0.21999573555161234,1.2809501403869654);
\draw[line width=1.5pt] (-0.21999573555161234,1.2809501403869654) -- (-0.20999576777381232,1.280264464197298);
\draw[line width=1.5pt] (-0.20999576777381232,1.280264464197298) -- (-0.1999957999960123,1.279610795785155);
\draw[line width=1.5pt] (-0.1999957999960123,1.279610795785155) -- (-0.18999583221821228,1.2789891185766797);
\draw[line width=1.5pt] (-0.18999583221821228,1.2789891185766797) -- (-0.17999586444041227,1.2783994168303812);
\draw[line width=1.5pt] (-0.17999586444041227,1.2783994168303812) -- (-0.16999589666261225,1.2778416756334983);
\draw[line width=1.5pt] (-0.16999589666261225,1.2778416756334983) -- (-0.15999592888481223,1.2773158808985685);
\draw[line width=1.5pt] (-0.15999592888481223,1.2773158808985685) -- (-0.14999596110701222,1.276822019360191);
\draw[line width=1.5pt] (-0.14999596110701222,1.276822019360191) -- (-0.1399959933292122,1.2763600785719873);
\draw[line width=1.5pt] (-0.1399959933292122,1.2763600785719873) -- (-0.12999602555141218,1.2759300469037553);
\draw[line width=1.5pt] (-0.12999602555141218,1.2759300469037553) -- (-0.11999605777361216,1.2755319135388155);
\draw[line width=1.5pt] (-0.11999605777361216,1.2755319135388155) -- (-0.10999608999581215,1.2751656684715464);
\draw[line width=1.5pt] (-0.10999608999581215,1.2751656684715464) -- (-0.09999612221801213,1.2748313025051106);
\draw[line width=1.5pt] (-0.09999612221801213,1.2748313025051106) -- (-0.08999615444021211,1.2745288072493661);
\draw[line width=1.5pt] (-0.08999615444021211,1.2745288072493661) -- (-0.0799961866624121,1.2742581751189637);
\draw[line width=1.5pt] (-0.0799961866624121,1.2742581751189637) -- (-0.06999621888461208,1.2740193993316327);
\draw[line width=1.5pt] (-0.06999621888461208,1.2740193993316327) -- (-0.05999625110681207,1.273812473906643);
\draw[line width=1.5pt] (-0.05999625110681207,1.273812473906643) -- (-0.049996283329012056,1.2736373936634588);
\draw[line width=1.5pt] (-0.049996283329012056,1.2736373936634588) -- (-0.039996315551212046,1.2734941542205662);
\draw[line width=1.5pt] (-0.039996315551212046,1.2734941542205662) -- (-0.029996347773412035,1.2733827519944867);
\draw[line width=1.5pt] (-0.029996347773412035,1.2733827519944867) -- (-0.019996379995612025,1.27330318419897);
\draw[line width=1.5pt] (-0.019996379995612025,1.27330318419897) -- (-0.009996412217812015,1.2732554488443653);
\draw[line width=1.5pt] (-0.009996412217812015,1.2732554488443653) -- (0.0,1.273239544737175);
\draw[line width=1.5pt] (0.0,1.273239544737175) -- (0.00999996777780001,1.273255460160063);
\draw[line width=1.5pt] (0.00999996777780001,1.273255460160063) -- (0.01999993555560002,1.27330320683266);
\draw[line width=1.5pt] (0.01999993555560002,1.27330320683266) -- (0.02999990333340003,1.2733827859467353);
\draw[line width=1.5pt] (0.02999990333340003,1.2733827859467353) -- (0.03999987111120004,1.2734941994922222);
\draw[line width=1.5pt] (0.03999987111120004,1.2734941994922222) -- (0.04999983888900005,1.2736374502556547);
\draw[line width=1.5pt] (0.04999983888900005,1.2736374502556547) -- (0.05999980666680006,1.273812541820795);
\draw[line width=1.5pt] (0.05999980666680006,1.273812541820795) -- (0.06999977444460007,1.2740194785694403);
\draw[line width=1.5pt] (0.06999977444460007,1.2740194785694403) -- (0.07999974222240008,1.2742582656824115);
\draw[line width=1.5pt] (0.07999974222240008,1.2742582656824115) -- (0.0899997100002001,1.2745289091407228);
\draw[line width=1.5pt] (0.0899997100002001,1.2745289091407228) -- (0.09999967777800012,1.2748314157269307);
\draw[line width=1.5pt] (0.09999967777800012,1.2748314157269307) -- (0.10999964555580014,1.2751657930266695);
\draw[line width=1.5pt] (0.10999964555580014,1.2751657930266695) -- (0.11999961333360015,1.2755320494303684);
\draw[line width=1.5pt] (0.11999961333360015,1.2755320494303684) -- (0.12999958111140017,1.275930194135152);
\draw[line width=1.5pt] (0.12999958111140017,1.275930194135152) -- (0.1399995488892002,1.2763602371469287);
\draw[line width=1.5pt] (0.1399995488892002,1.2763602371469287) -- (0.1499995166670002,1.2768221892826683);
\draw[line width=1.5pt] (0.1499995166670002,1.2768221892826683) -- (0.15999948444480022,1.2773160621728616);
\draw[line width=1.5pt] (0.15999948444480022,1.2773160621728616) -- (0.16999945222260024,1.2778418682641783);
\draw[line width=1.5pt] (0.16999945222260024,1.2778418682641783) -- (0.17999942000040026,1.2783996208223103);
\draw[line width=1.5pt] (0.17999942000040026,1.2783996208223103) -- (0.18999938777820027,1.2789893339350138);
\draw[line width=1.5pt] (0.18999938777820027,1.2789893339350138) -- (0.1999993555560003,1.279611022515343);
\draw[line width=1.5pt] (0.1999993555560003,1.279611022515343) -- (0.2099993233338003,1.2802647023050855);
\draw[line width=1.5pt] (0.2099993233338003,1.2802647023050855) -- (0.21999929111160033,1.2809503898783932);
\draw[line width=1.5pt] (0.21999929111160033,1.2809503898783932) -- (0.22999925888940034,1.2816681026456211);
\draw[line width=1.5pt] (0.22999925888940034,1.2816681026456211) -- (0.23999922666720036,1.2824178588573651);
\draw[line width=1.5pt] (0.23999922666720036,1.2824178588573651) -- (0.24999919444500038,1.2831996776087131);
\draw[line width=1.5pt] (0.24999919444500038,1.2831996776087131) -- (0.2599991622228004,1.284013578843701);
\draw[line width=1.5pt] (0.2599991622228004,1.284013578843701) -- (0.2699991300006004,1.2848595833599876);
\draw[line width=1.5pt] (0.2699991300006004,1.2848595833599876) -- (0.2799990977784004,1.2857377128137393);
\draw[line width=1.5pt] (0.2799990977784004,1.2857377128137393) -- (0.28999906555620036,1.2866479897247414);
\draw[line width=1.5pt] (0.28999906555620036,1.2866479897247414) -- (0.29999903333400035,1.2875904374817273);
\draw[line width=1.5pt] (0.29999903333400035,1.2875904374817273) -- (0.30999900111180034,1.2885650803479367);
\draw[line width=1.5pt] (0.30999900111180034,1.2885650803479367) -- (0.31999896888960033,1.2895719434669033);
\draw[line width=1.5pt] (0.31999896888960033,1.2895719434669033) -- (0.3299989366674003,1.2906110528684782);
\draw[line width=1.5pt] (0.3299989366674003,1.2906110528684782) -- (0.3399989044452003,1.2916824354750898);
\draw[line width=1.5pt] (0.3399989044452003,1.2916824354750898) -- (0.3499988722230003,1.2927861191082481);
\draw[line width=1.5pt] (0.3499988722230003,1.2927861191082481) -- (0.3599988400008003,1.2939221324952963);
\draw[line width=1.5pt] (0.3599988400008003,1.2939221324952963) -- (0.3699988077786003,1.2950905052764123);
\draw[line width=1.5pt] (0.3699988077786003,1.2950905052764123) -- (0.37999877555640027,1.2962912680118712);
\draw[line width=1.5pt] (0.37999877555640027,1.2962912680118712) -- (0.38999874333420026,1.2975244521895664);
\draw[line width=1.5pt] (0.38999874333420026,1.2975244521895664) -- (0.39999871111200025,1.2987900902327978);
\draw[line width=1.5pt] (0.39999871111200025,1.2987900902327978) -- (0.40999867888980024,1.3000882155083375);
\draw[line width=1.5pt] (0.40999867888980024,1.3000882155083375) -- (0.41999864666760023,1.3014188623347702);
\draw[line width=1.5pt] (0.41999864666760023,1.3014188623347702) -- (0.4299986144454002,1.302782065991118);
\draw[line width=1.5pt] (0.4299986144454002,1.302782065991118) -- (0.4399985822232002,1.3041778627257623);
\draw[line width=1.5pt] (0.4399985822232002,1.3041778627257623) -- (0.4499985500010002,1.305606289765656);
\draw[line width=1.5pt] (0.4499985500010002,1.305606289765656) -- (0.4599985177788002,1.3070673853258459);
\draw[line width=1.5pt] (0.4599985177788002,1.3070673853258459) -- (0.4699984855566002,1.3085611886193058);
\draw[line width=1.5pt] (0.4699984855566002,1.3085611886193058) -- (0.47999845333440017,1.310087739867088);
\draw[line width=1.5pt] (0.47999845333440017,1.310087739867088) -- (0.48999842111220016,1.3116470803088016);
\draw[line width=1.5pt] (0.48999842111220016,1.3116470803088016) -- (0.49999838889000014,1.3132392522134297);
\draw[line width=1.5pt] (0.49999838889000014,1.3132392522134297) -- (0.5099983566678001,1.314864298890485);
\draw[line width=1.5pt] (0.5099983566678001,1.314864298890485) -- (0.5199983244456001,1.3165222647015231);
\draw[line width=1.5pt] (0.5199983244456001,1.3165222647015231) -- (0.5299982922234001,1.3182131950720128);
\draw[line width=1.5pt] (0.5299982922234001,1.3182131950720128) -- (0.5399982600012001,1.3199371365035812);
\draw[line width=1.5pt] (0.5399982600012001,1.3199371365035812) -- (0.5499982277790001,1.3216941365866388);
\draw[line width=1.5pt] (0.5499982277790001,1.3216941365866388) -- (0.5599981955568001,1.3234842440133947);
\draw[line width=1.5pt] (0.5599981955568001,1.3234842440133947) -- (0.5699981633346001,1.3253075085912756);
\draw[line width=1.5pt] (0.5699981633346001,1.3253075085912756) -- (0.5799981311124001,1.3271639812567577);
\draw[line width=1.5pt] (0.5799981311124001,1.3271639812567577) -- (0.5899980988902,1.329053714089623);
\draw[line width=1.5pt] (0.5899980988902,1.329053714089623) -- (0.599998066668,1.3309767603276548);
\draw[line width=1.5pt] (0.599998066668,1.3309767603276548) -- (0.6099980344458,1.332933174381779);
\draw[line width=1.5pt] (0.6099980344458,1.332933174381779) -- (0.6199980022236,1.3349230118516742);
\draw[line width=1.5pt] (0.6199980022236,1.3349230118516742) -- (0.6299979700014,1.336946329541853);
\draw[line width=1.5pt] (0.6299979700014,1.336946329541853) -- (0.6399979377792,1.3390031854782392);
\draw[line width=1.5pt] (0.6399979377792,1.3390031854782392) -- (0.649997905557,1.3410936389252461);
\draw[line width=1.5pt] (0.649997905557,1.3410936389252461) -- (0.6599978733348,1.3432177504033826);
\draw[line width=1.5pt] (0.6599978733348,1.3432177504033826) -- (0.6699978411126,1.345375581707391);
\draw[line width=1.5pt] (0.6699978411126,1.345375581707391) -- (0.6799978088904,1.3475671959249442);
\draw[line width=1.5pt] (0.6799978088904,1.3475671959249442) -- (0.6899977766682,1.349792657455914);
\draw[line width=1.5pt] (0.6899977766682,1.349792657455914) -- (0.6999977444459999,1.3520520320322291);
\draw[line width=1.5pt] (0.6999977444459999,1.3520520320322291) -- (0.7099977122237999,1.354345386738346);
\draw[line width=1.5pt] (0.7099977122237999,1.354345386738346) -- (0.7199976800015999,1.3566727900323476);
\draw[line width=1.5pt] (0.7199976800015999,1.3566727900323476) -- (0.7299976477793999,1.3590343117676966);
\draw[line width=1.5pt] (0.7299976477793999,1.3590343117676966) -- (0.7399976155571999,1.3614300232156575);
\draw[line width=1.5pt] (0.7399976155571999,1.3614300232156575) -- (0.7499975833349999,1.3638599970884202);
\draw[line width=1.5pt] (0.7499975833349999,1.3638599970884202) -- (0.7599975511127999,1.3663243075629399);
\draw[line width=1.5pt] (0.7599975511127999,1.3663243075629399) -- (0.7699975188905999,1.3688230303055227);
\draw[line width=1.5pt] (0.7699975188905999,1.3688230303055227) -- (0.7799974866683999,1.3713562424971861);
\draw[line width=1.5pt] (0.7799974866683999,1.3713562424971861) -- (0.7899974544461998,1.3739240228598137);
\draw[line width=1.5pt] (0.7899974544461998,1.3739240228598137) -- (0.7999974222239998,1.3765264516831404);
\draw[line width=1.5pt] (0.7999974222239998,1.3765264516831404) -- (0.8099973900017998,1.3791636108525935);
\draw[line width=1.5pt] (0.8099973900017998,1.3791636108525935) -- (0.8199973577795998,1.3818355838780252);
\draw[line width=1.5pt] (0.8199973577795998,1.3818355838780252) -- (0.8299973255573998,1.3845424559233623);
\draw[line width=1.5pt] (0.8299973255573998,1.3845424559233623) -- (0.8399972933351998,1.3872843138372177);
\draw[line width=1.5pt] (0.8399972933351998,1.3872843138372177) -- (0.8499972611129998,1.3900612461844903);
\draw[line width=1.5pt] (0.8499972611129998,1.3900612461844903) -- (0.8599972288907998,1.3928733432789957);
\draw[line width=1.5pt] (0.8599972288907998,1.3928733432789957) -- (0.8699971966685998,1.3957206972171696);
\draw[line width=1.5pt] (0.8699971966685998,1.3957206972171696) -- (0.8799971644463997,1.3986034019128788);
\draw[line width=1.5pt] (0.8799971644463997,1.3986034019128788) -- (0.8899971322241997,1.40152155313339);
\draw[line width=1.5pt] (0.8899971322241997,1.40152155313339) -- (0.8999971000019997,1.4044752485365402);
\draw[line width=1.5pt] (0.8999971000019997,1.4044752485365402) -- (0.9099970677797997,1.4074645877091543);
\draw[line width=1.5pt] (0.9099970677797997,1.4074645877091543) -- (0.9199970355575997,1.4104896722067661);
\draw[line width=1.5pt] (0.9199970355575997,1.4104896722067661) -- (0.9299970033353997,1.4135506055946896);
\draw[line width=1.5pt] (0.9299970033353997,1.4135506055946896) -- (0.9399969711131997,1.4166474934905016);
\draw[line width=1.5pt] (0.9399969711131997,1.4166474934905016) -- (0.9499969388909997,1.4197804436079928);
\draw[line width=1.5pt] (0.9499969388909997,1.4197804436079928) -- (0.9599969066687997,1.422949565802649);
\draw[line width=1.5pt] (0.9599969066687997,1.422949565802649) -- (0.9699968744465997,1.4261549721187279);
\draw[line width=1.5pt] (0.9699968744465997,1.4261549721187279) -- (0.9799968422243996,1.4293967768380025);
\draw[line width=1.5pt] (0.9799968422243996,1.4293967768380025) -- (0.9899968100021996,1.4326750965302386);
\draw[line width=1.5pt] (0.9899968100021996,1.4326750965302386) -- (0.9999967777799996,1.4359900501054923);
\draw[line width=1.5pt] (0.9999967777799996,1.4359900501054923) -- (1.0099967455577996,1.4393417588682935);
\draw[line width=1.5pt] (1.0099967455577996,1.4393417588682935) -- (1.0199967133355996,1.4427303465738202);
\draw[line width=1.5pt] (1.0199967133355996,1.4427303465738202) -- (1.0299966811133996,1.4461559394861332);
\draw[line width=1.5pt] (1.0299966811133996,1.4461559394861332) -- (1.0399966488911996,1.4496186664385846);
\draw[line width=1.5pt] (1.0399966488911996,1.4496186664385846) -- (1.0499966166689996,1.453118658896487);
\draw[line width=1.5pt] (1.0499966166689996,1.453118658896487) -- (1.0599965844467996,1.4566560510221598);
\draw[line width=1.5pt] (1.0599965844467996,1.4566560510221598) -- (1.0699965522245996,1.4602309797424562);
\draw[line width=1.5pt] (1.0699965522245996,1.4602309797424562) -- (1.0799965200023995,1.4638435848189029);
\draw[line width=1.5pt] (1.0799965200023995,1.4638435848189029) -- (1.0899964877801995,1.467494008920564);
\draw[line width=1.5pt] (1.0899964877801995,1.467494008920564) -- (1.0999964555579995,1.4711823976997787);
\draw[line width=1.5pt] (1.0999964555579995,1.4711823976997787) -- (1.1099964233357995,1.4749088998708997);
\draw[line width=1.5pt] (1.1099964233357995,1.4749088998708997) -- (1.1199963911135995,1.4786736672922005);
\draw[line width=1.5pt] (1.1199963911135995,1.4786736672922005) -- (1.1299963588913995,1.4824768550510987);
\draw[line width=1.5pt] (1.1299963588913995,1.4824768550510987) -- (1.1399963266691995,1.486318621552871);
\draw[line width=1.5pt] (1.1399963266691995,1.486318621552871) -- (1.1499962944469995,1.4901991286130467);
\draw[line width=1.5pt] (1.1499962944469995,1.4901991286130467) -- (1.1599962622247995,1.4941185415536664);
\draw[line width=1.5pt] (1.1599962622247995,1.4941185415536664) -- (1.1699962300025994,1.4980770293036132);
\draw[line width=1.5pt] (1.1699962300025994,1.4980770293036132) -- (1.1799961977803994,1.5020747645032435);
\draw[line width=1.5pt] (1.1799961977803994,1.5020747645032435) -- (1.1899961655581994,1.506111923613544);
\draw[line width=1.5pt] (1.1899961655581994,1.506111923613544) -- (1.1999961333359994,1.51018868703007);
\draw[line width=1.5pt] (1.1999961333359994,1.51018868703007) -- (1.2099961011137994,1.5143052392019307);
\draw[line width=1.5pt] (1.2099961011137994,1.5143052392019307) -- (1.2199960688915994,1.5184617687561126);
\draw[line width=1.5pt] (1.2199960688915994,1.5184617687561126) -- (1.2299960366693994,1.522658468627444);
\draw[line width=1.5pt] (1.2299960366693994,1.522658468627444) -- (1.2399960044471994,1.5268955361945287);
\draw[line width=1.5pt] (1.2399960044471994,1.5268955361945287) -- (1.2499959722249994,1.531173173422007);
\draw[line width=1.5pt] (1.2499959722249994,1.531173173422007) -- (1.2599959400027994,1.5354915870095176);
\draw[line width=1.5pt] (1.2599959400027994,1.5354915870095176) -- (1.2699959077805993,1.5398509885477658);
\draw[line width=1.5pt] (1.2699959077805993,1.5398509885477658) -- (1.2799958755583993,1.5442515946821431);
\draw[line width=1.5pt] (1.2799958755583993,1.5442515946821431) -- (1.2899958433361993,1.54869362728436);
\draw[line width=1.5pt] (1.2899958433361993,1.54869362728436) -- (1.2999958111139993,1.5531773136326041);
\draw[line width=1.5pt] (1.2999958111139993,1.5531773136326041) -- (1.3099957788917993,1.5577028866007685);
\draw[line width=1.5pt] (1.3099957788917993,1.5577028866007685) -- (1.3199957466695993,1.5622705848573404);
\draw[line width=1.5pt] (1.3199957466695993,1.5622705848573404) -- (1.3299957144473993,1.5668806530745925);
\draw[line width=1.5pt] (1.3299957144473993,1.5668806530745925) -- (1.3399956822251993,1.5715333421487574);
\draw[line width=1.5pt] (1.3399956822251993,1.5715333421487574) -- (1.3499956500029993,1.5762289094319464);
\draw[line width=1.5pt] (1.3499956500029993,1.5762289094319464) -- (1.3599956177807992,1.5809676189766104);
\draw[line width=1.5pt] (1.3599956177807992,1.5809676189766104) -- (1.3699955855585992,1.5857497417934339);
\draw[line width=1.5pt] (1.3699955855585992,1.5857497417934339) -- (1.3799955533363992,1.5905755561236095);
\draw[line width=1.5pt] (1.3799955533363992,1.5905755561236095) -- (1.3899955211141992,1.5954453477265476);
\draw[line width=1.5pt] (1.3899955211141992,1.5954453477265476) -- (1.3999954888919992,1.6003594101841419);
\draw[line width=1.5pt] (1.3999954888919992,1.6003594101841419) -- (1.4099954566697992,1.6053180452228424);
\draw[line width=1.5pt] (1.4099954566697992,1.6053180452228424) -- (1.4199954244475992,1.6103215630548848);
\draw[line width=1.5pt] (1.4199954244475992,1.6103215630548848) -- (1.4299953922253992,1.615370282740162);
\draw[line width=1.5pt] (1.4299953922253992,1.615370282740162) -- (1.4399953600031992,1.6204645325703588);
\draw[line width=1.5pt] (1.4399953600031992,1.6204645325703588) -- (1.4499953277809992,1.6256046504771446);
\draw[line width=1.5pt] (1.4499953277809992,1.6256046504771446) -- (1.4599952955587991,1.6307909844663748);
\draw[line width=1.5pt] (1.4599952955587991,1.6307909844663748) -- (1.4699952633365991,1.6360238930804725);
\draw[line width=1.5pt] (1.4699952633365991,1.6360238930804725) -- (1.4799952311143991,1.6413037458913737);
\draw[line width=1.5pt] (1.4799952311143991,1.6413037458913737) -- (1.4899951988921991,1.6466309240266748);
\draw[line width=1.5pt] (1.4899951988921991,1.6466309240266748) -- (1.499995166669999,1.6520058207319004);
\draw[line width=1.5pt] (1.499995166669999,1.6520058207319004) -- (1.509995134447799,1.6574288419721408);
\draw[line width=1.5pt] (1.509995134447799,1.6574288419721408) -- (1.519995102225599,1.6629004070766549);
\draw[line width=1.5pt] (1.519995102225599,1.6629004070766549) -- (1.529995070003399,1.6684209494304623);
\draw[line width=1.5pt] (1.529995070003399,1.6684209494304623) -- (1.539995037781199,1.6739909172174032);
\draw[line width=1.5pt] (1.539995037781199,1.6739909172174032) -- (1.549995005558999,1.6796107742196922);
\draw[line width=1.5pt] (1.549995005558999,1.6796107742196922) -- (1.559994973336799,1.6852810006795862);
\draw[line width=1.5pt] (1.559994973336799,1.6852810006795862) -- (1.569994941114599,1.6910020942295105);
\draw[line width=1.5pt] (1.569994941114599,1.6910020942295105) -- (1.579994908892399,1.6967745708977642);
\draw[line width=1.5pt] (1.579994908892399,1.6967745708977642) -- (1.589994876670199,1.7025989661978935);
\draw[line width=1.5pt] (1.589994876670199,1.7025989661978935) -- (1.599994844447999,1.7084758363108599);
\draw[line width=1.5pt] (1.599994844447999,1.7084758363108599) -- (1.609994812225799,1.7144057593704174);
\draw[line width=1.5pt] (1.609994812225799,1.7144057593704174) -- (1.619994780003599,1.7203893368635457);
\draw[line width=1.5pt] (1.619994780003599,1.7203893368635457) -- (1.629994747781399,1.7264271951595167);
\draw[line width=1.5pt] (1.629994747781399,1.7264271951595167) -- (1.639994715559199,1.7325199871831574);
\draw[line width=1.5pt] (1.639994715559199,1.7325199871831574) -- (1.649994683336999,1.7386683942502614);
\draw[line width=1.5pt] (1.649994683336999,1.7386683942502614) -- (1.659994651114799,1.7448731280858876);
\draw[line width=1.5pt] (1.659994651114799,1.7448731280858876) -- (1.669994618892599,1.7511349330496346);
\draw[line width=1.5pt] (1.669994618892599,1.7511349330496346) -- (1.679994586670399,1.7574545885959665);
\draw[line width=1.5pt] (1.679994586670399,1.7574545885959665) -- (1.689994554448199,1.7638329120024663);
\draw[line width=1.5pt] (1.689994554448199,1.7638329120024663) -- (1.699994522225999,1.7702707614046882);
\draw[line width=1.5pt] (1.699994522225999,1.7702707614046882) -- (1.7099944900037989,1.7767690391833348);
\draw[line width=1.5pt] (1.7099944900037989,1.7767690391833348) -- (1.7199944577815989,1.7833286957580894);
\draw[line width=1.5pt] (1.7199944577815989,1.7833286957580894) -- (1.7299944255593989,1.7899507338530511);
\draw[line width=1.5pt] (1.7299944255593989,1.7899507338530511) -- (1.7399943933371989,1.7966362133118507);
\draw[line width=1.5pt] (1.7399943933371989,1.7966362133118507) -- (1.7499943611149988,1.803386256556929);
\draw[line width=1.5pt] (1.7499943611149988,1.803386256556929) -- (1.7599943288927988,1.8102020548080957);
\draw[line width=1.5pt] (1.7599943288927988,1.8102020548080957) -- (1.7699942966705988,1.8170848752016706);
\draw[line width=1.5pt] (1.7699942966705988,1.8170848752016706) -- (1.7799942644483988,1.8240360689850195);
\draw[line width=1.5pt] (1.7799942644483988,1.8240360689850195) -- (1.7899942322261988,1.8310570810046698);
\draw[line width=1.5pt] (1.7899942322261988,1.8310570810046698) -- (1.7999942000039988,1.8381494607628222);
\draw[line width=1.5pt] (1.7999942000039988,1.8381494607628222) -- (1.8099941677817988,1.8453148753919935);
\draw[line width=1.5pt] (1.8099941677817988,1.8453148753919935) -- (1.8199941355595988,1.8525551249978347);
\draw[line width=1.5pt] (1.8199941355595988,1.8525551249978347) -- (1.8299941033373988,1.8598721609563478);
\draw[line width=1.5pt] (1.8299941033373988,1.8598721609563478) -- (1.8399940711151987,1.8672681079394833);
\draw[line width=1.5pt] (1.8399940711151987,1.8672681079394833) -- (1.8499940388929987,1.8747452907063695);
\draw[line width=1.5pt] (1.8499940388929987,1.8747452907063695) -- (1.8599940066707987,1.8823062670735522);
\draw[line width=1.5pt] (1.8599940066707987,1.8823062670735522) -- (1.8699939744485987,1.889953869026515);
\draw[line width=1.5pt] (1.8699939744485987,1.889953869026515) -- (1.8799939422263987,1.8976912547549154);
\draw[line width=1.5pt] (1.8799939422263987,1.8976912547549154) -- (1.8899939100041987,1.9055219756531863);
\draw[line width=1.5pt] (1.8899939100041987,1.9055219756531863) -- (1.8999938777819987,1.9134500643222767);
\draw[line width=1.5pt] (1.8999938777819987,1.9134500643222767) -- (1.9099938455597987,1.9214801528821597);
\draw[line width=1.5pt] (1.9099938455597987,1.9214801528821597) -- (1.9199938133375987,1.9296176365122901);
\draw[line width=1.5pt] (1.9199938133375987,1.9296176365122901) -- (1.9299937811153987,1.937868907240598);
\draw[line width=1.5pt] (1.9299937811153987,1.937868907240598) -- (1.9399937488931986,1.9462417023643837);
\draw[line width=1.5pt] (1.9399937488931986,1.9462417023643837) -- (1.9499937166709986,1.9547456519684403);
\draw[line width=1.5pt] (1.9499937166709986,1.9547456519684403) -- (1.9599936844487986,1.9633932016702904);
\draw[line width=1.5pt] (1.9599936844487986,1.9633932016702904) -- (1.9699936522265986,1.9722013266534018);
\draw[line width=1.5pt] (1.9699936522265986,1.9722013266534018) -- (1.9799936200043986,1.9811952178640992);
\draw[line width=1.5pt] (1.9799936200043986,1.9811952178640992) -- (1.9899935877821986,1.9904185156026915);
\draw[line width=1.5pt] (1.9899935877821986,1.9904185156026915) -- (1.9999935555599986,1.999993562503313);
\draw [line width=0.8pt] (-1.575,1.725)-- (0.075,0.075);
\draw [line width=0.8pt] (-1.65,1.65)-- (-1.575,1.725);
\draw [line width=0.8pt] (1.575,1.725)-- (1.65,1.65);
\draw [line width=0.8pt] (1.575,1.725)-- (-0.075,0.075);
\draw [line width=0.8pt] (-1.2,1.5)-- (0.15,0.15);
\draw [line width=0.8pt] (1.275,1.575)-- (-0.15,0.15);
\draw [line width=0.8pt] (-1.05,1.5)-- (0.225,0.225);
\draw [line width=0.8pt] (0.975,1.425)-- (-0.225,0.225);
\draw [line width=0.8pt] (-0.75,1.35)-- (0.3,0.3);
\draw [line width=0.8pt] (0.75,1.35)-- (-0.3,0.3);
\draw [line width=0.8pt] (0.675,1.425)-- (-0.375,0.375);
\draw [line width=0.8pt] (-0.6,1.35)-- (0.375,0.375);
\draw [line width=0.8pt] (0.375,1.275)-- (-0.45,0.45);
\draw [line width=0.8pt] (0.225,1.275)-- (-0.525,0.525);
\draw [line width=0.8pt] (0.15,1.35)-- (-0.6,0.6);
\draw [line width=0.8pt] (-0.07411385684356032,1.274113861238779)-- (0.6,0.6);
\draw [line width=0.8pt] (-1.2,1.5)-- (-1.35,1.35);
\draw [line width=0.8pt] (-1.05,1.5)-- (-1.275,1.275);
\draw [line width=0.8pt] (-0.75,1.35)-- (-1.05,1.05);
\draw [line width=0.8pt] (-0.6,1.35)-- (-0.975,0.975);
\draw [line width=0.8pt] (-0.375,1.275)-- (-0.825,0.825);
\draw [line width=0.8pt] (-0.375,1.275)-- (-0.3,1.35);
\draw [line width=0.8pt] (-0.07411385684356032,1.274113861238779)-- (-0.675,0.675);
\draw [line width=0.8pt] (0.15,1.35)-- (0.75,0.75);
\draw [line width=0.8pt] (0.375,1.275)-- (0.825,0.825);
\draw [line width=0.8pt] (0.675,1.425)-- (1.05,1.05);
\draw [line width=0.8pt] (0.975,1.425)-- (1.2,1.2);
\draw [line width=0.8pt] (1.275,1.575)-- (1.425,1.425);
\draw [line width=0.8pt] (-1.575,1.575)-- (-1.5,1.65);
\draw [line width=0.8pt] (-1.425,1.575)-- (-1.5,1.5);
\draw [line width=0.8pt] (-1.35,1.5)-- (-1.425,1.425);
\draw [line width=0.8pt] (-0.975,1.425)-- (-1.2,1.2);
\draw [line width=0.8pt] (-0.9,1.35)-- (-1.125,1.125);
\draw [line width=0.8pt] (-0.225,1.275)-- (-0.75,0.75);
\draw [line width=0.8pt] (0.075,1.275)-- (0.675,0.675);
\draw [line width=0.8pt] (0.525,1.275)-- (0.9,0.9);
\draw [line width=0.8pt] (0.6,1.35)-- (0.975,0.975);
\draw [line width=0.8pt] (0.9,1.35)-- (1.125,1.125);
\draw [line width=0.8pt] (1.125,1.425)-- (1.275,1.275);
\draw [line width=0.8pt] (1.2,1.5)-- (1.35,1.35);
\draw [line width=0.8pt] (1.425,1.575)-- (1.5,1.5);
\draw [line width=0.8pt] (1.5,1.65)-- (1.575,1.575);
\draw [line width=0.8pt] (-0.9,0.9)-- (-0.375,1.425);
\draw [line width=0.8pt] (-0.45,1.35)-- (0.45,0.45);
\draw [line width=0.8pt] (-0.375,1.425)-- (0.525,0.525);
\draw [shift={(-5.,0.)},line width=0.8pt]  plot[domain=0.:3.141592653589793,variable=\t]({1.*2.*cos(\t r)+0.*2.*sin(\t r)},{0.*2.*cos(\t r)+1.*2.*sin(\t r)});
\draw [line width=1.5pt] (-7.3,0.)-- (-2.7,0.);
\draw [line width=1.5pt] (-2.5,2.5)-- (0.,0.);
\draw [line width=1.5pt] (0.,0.)-- (2.5,2.5);
\draw [line width=1.5pt] (-5.,2.5)-- (-5.,-0.5);
\begin{scriptsize}
\draw [fill=uuuuuu] (-3.,0.) circle (2.0pt);
\draw[color=uuuuuu] (-3,-0.25) node {$(2,0)$};
\draw [fill=uuuuuu] (-7.,0.) circle (2.0pt);
\draw[color=uuuuuu] (-7, -0.25) node {$(-2,0)$};
\draw [fill=uuuuuu] (-5.,0.) circle (2.0pt);
\draw[color=uuuuuu] (-4.7,0.25) node {$(0,0)$};
\draw [fill=uuuuuu] (-5.,2.) circle (2.0pt);
\draw[color=uuuuuu] (-4.65, 2.25) node {$\left(0, \frac{1}{\pi}\right)$};
\end{scriptsize}
\end{tikzpicture}
\end{center}

%% file: Interlacing.tex
\begin{center}
\begin{tikzpicture}[line cap=round,line join=round,>=triangle 45,x=2.4cm,y=2.4cm]
\clip(-2.5,-0.2) rectangle (2.5,2.5);
\draw [line width=1.5pt,domain=-2.5:-1.65] plot(\x,{(-0.--2.*\x)/-2.});
\draw [line width=0.8pt,domain=-1.65:0.0] plot(\x,{(-0.--2.*\x)/-2.});
\draw [line width=0.8pt,domain=0.0:1.65] plot(\x,{(-0.--2.*\x)/2.});
\draw [line width=1.5pt,domain=1.65:2.5] plot(\x,{(-0.--2.*\x)/2.});
\draw [line width=0.8pt,domain=-2.5:2.5] plot(\x,{(-0.-0.*\x)/1.});
\draw [line width=1.5pt] (-1.65,1.65)-- (-1.5,1.8);
\draw [line width=1.5pt] (-1.5,1.8)-- (-1.125,1.425);
\draw [line width=1.5pt] (-1.5,1.8)-- (-1.65,1.65);
\draw [line width=1.5pt] (-1.125,1.425)-- (-0.9,1.65);
\draw [line width=1.5pt] (-0.525,1.275)-- (-0.9,1.65);
\draw [line width=1.5pt] (-0.3,1.5)-- (-0.525,1.275);
\draw [line width=1.5pt] (-0.3,1.5)-- (0.225,0.975);
\draw [line width=1.5pt] (0.225,0.975)-- (0.525,1.275);
\draw [line width=1.5pt] (0.525,1.275)-- (0.675,1.125);
\draw [line width=1.5pt] (1.425,1.875)-- (1.65,1.65);
\draw [line width=1.5pt] (0.675,1.125)-- (1.425,1.875);
\draw [line width=0.8pt,dash pattern=on 3pt off 3pt] (-1.65,0.) -- (-1.65,2.5);
\draw [line width=0.8pt,dash pattern=on 3pt off 3pt] (-1.5,0.) -- (-1.5,2.5);
\draw [line width=0.8pt,dash pattern=on 3pt off 3pt] (-1.125,0.) -- (-1.125,2.5);
\draw [line width=0.8pt,dash pattern=on 3pt off 3pt] (-0.9,0.) -- (-0.9,2.5);
\draw [line width=0.8pt,dash pattern=on 3pt off 3pt] (-0.525,0.) -- (-0.525,2.5);
\draw [line width=0.8pt,dash pattern=on 3pt off 3pt] (-0.3,0.) -- (-0.3,2.5);
\draw [line width=0.8pt,dash pattern=on 3pt off 3pt] (0.225,0.) -- (0.225,2.5);
\draw [line width=0.8pt,dash pattern=on 3pt off 3pt] (0.525,0.) -- (0.525,2.5);
\draw [line width=0.8pt,dash pattern=on 3pt off 3pt] (0.675,0.) -- (0.675,2.5);
\draw [line width=0.8pt,dash pattern=on 3pt off 3pt] (1.425,0.) -- (1.425,2.5);
\draw [line width=0.8pt,dash pattern=on 3pt off 3pt] (1.65,0.) -- (1.65,2.5);
\begin{scriptsize}
\draw [fill=black] (-1.65,0.) circle (1.8pt);
\draw[color=black] (-1.65, -0.125) node {$x_1$};
\draw [fill=black] (-1.5,0.) circle (1.8pt);
\draw[color=black] (-1.5,-0.125) node {$y_1$};
\draw [fill=black] (-1.125,0.) circle (1.8pt);
\draw[color=black] (-1.125,-0.125) node {$x_2$};
\draw [fill=black] (-0.9,0.) circle (1.8pt);
\draw[color=black] (-0.9,-0.125) node {$y_2$};
\draw [fill=black] (-0.525,0.) circle (1.8pt);
\draw[color=black] (-0.525,-0.125) node {$x_3$};
\draw [fill=black] (-0.3,0.) circle (1.8pt);
\draw[color=black] (-0.3,-0.125) node {$y_3$};
\draw [fill=black] (0.225,0.) circle (1.8pt);
\draw[color=black] (0.225,-0.125) node {$x_4$};
\draw [fill=black] (0.525,0.) circle (1.8pt);
\draw[color=black] (0.525,-0.125) node {$y_4$};
\draw [fill=black] (0.675,0.) circle (1.8pt);
\draw[color=black] (0.675,-0.125) node {$x_5$};
\draw [fill=black] (1.425,0.) circle (1.8pt);
\draw[color=black] (1.425,-0.125) node {$y_5$};
\draw [fill=black] (1.65,0.) circle (1.8pt);
\draw[color=black] (1.65,-0.125) node {$x_6$};
\end{scriptsize}
\end{tikzpicture}
\end{center}